\documentclass[11pt,a4paper]{article}
\usepackage{amsmath,amssymb,amsthm,mathtools}
\usepackage{mathrsfs}
\usepackage{graphicx}
\usepackage{xcolor}
\usepackage{hyperref}
\usepackage{tikz-cd}
\usepackage{geometry}
\hypersetup{
  colorlinks=true,
  linkcolor=blue,
  citecolor=blue,
  urlcolor=blue,
  pdftitle={Mersenne Representation, the Conolly Sequence,
    and Soliton Profiles over Finite Fields},
  pdfauthor={Fumitaka Yura},
  pdfsubject={Mersenne representation and finite-field BBS},
  pdfkeywords={Mersenne representation, Conolly sequence,
    tau function, finite-field BBS, soliton}
}
\definecolor{binarygreen}{RGB}{0,105,70}
\definecolor{nonbinaryred}{RGB}{180,35,45}
\newtheorem{theorem}{Theorem}[section]
\newtheorem{corollary}[theorem]{Corollary}

\newtheorem{lemma}[theorem]{Lemma}

\newtheorem{proposition}[theorem]{Proposition}

\theoremstyle{definition}
\newtheorem{definition}[theorem]{Definition}
\newtheorem{remark}[theorem]{Remark}
\newtheorem{example}[theorem]{Example}
\newcommand{\Z}{\mathbb{Z}}

\newcommand{\F}{\mathbb{F}}
\newcommand{\LLang}{\mathcal{L}}
\newcommand{\eps}{\varepsilon}
\newcommand{\Afive}{\mathrm{A005187}}
\newcommand{\Afivefive}{\mathrm{A055938}}
\newcommand{\Aeight}{\mathrm{A080578}}
\newcommand{\Acon}{\mathrm{A046699}}
\newcommand{\Aonefiveoneone}{\mathrm{A001511}}
\newcommand{\Aoneonethreesevenone}{\mathrm{A011371}}
\newcommand{\Asevenninefivefivenine}{\mathrm{A079559}}
\title{Mersenne Representation, the Conolly Sequence, and Soliton Profiles over Finite Fields}
\author{Fumitaka YURA
\thanks{
Department of Mathematical Engineering,
Faculty of Engineering,
Musashino University,
3-3-3 Ariake,
Koto-ku,
Tokyo 135-8181,
Japan. 
E-mail: \texttt{f-yura@musashino-u.ac.jp}.
}}
\date{}

\begin{document}

\maketitle

\begin{abstract}
We study the canonical skew-binary representation of nonnegative integers, whose positional weights are the Mersenne numbers \(2^{k}-1\), and its decomposition into the binary and nonbinary sides. The nonbinary values, arranged in increasing order, form the OEIS sequence A055938, and their successor structure gives a direct proof of the relation with the OEIS sequence A080578 that is recorded as conjectural in the OEIS entry for A080578. The resulting formulas show that the Conolly sequence describes both sides of the decomposition: a shifted Conolly sequence is the counting function of the binary side, while the same sequence gives the increasing enumeration of the nonbinary side.

We then develop the parent map, truncation blocks, truncation remainders, and the Mersenne tau function associated with this representation. The parent map is conjugate to deletion of the lowest digit. Differences of the Mersenne tau rows recover the parent iterates and the counting function, and they give formulas for digit reconstruction and for a diagonal tau defect.

Finally, we revisit a known finite-depth one-soliton family of the finite-field BBS. The Mersenne combinatorics is used directly to reconstruct a global integer-valued traveling-wave profile and to prove an integer window-counting theorem. Reduction modulo \(3\) yields the corresponding finite-field traveling-wave solutions. We also construct an integer-valued traveling-wave tau function whose front values are given by the Mersenne tau rows. The resulting construction shows how the combinatorics of a number representation can enter directly into the reconstruction of solutions of an integrable system.
\end{abstract}

\medskip


\medskip


\section{Introduction}

We consider the canonical skew-binary representation introduced by Myers \cite{Myers1983}. Its positional weights are the Mersenne numbers \(2^{k}-1\), and its admissible words use the digits \(\mathtt{0}\), \(\mathtt{1}\), and \(\mathtt{2}\), where the digit \(\mathtt{2}\) forces all lower digits to be \(\mathtt{0}\). After unnecessary leading zeros are removed, every nonnegative integer has a unique admissible word. To emphasize the positional weights, we call this the Mersenne representation in the present paper.

Canonical skew-binary representations have also been used in purely functional data structures, including random-access stacks and random-access lists \cite{Myers1983,Okasaki1995}. More generally, the relation between numeral systems and data structures has been developed by Elmasry, Jensen, and Katajainen \cite{ElmasryJensenKatajainen2012}.

Our purpose is not to reintroduce the underlying canonical representation, but to develop its binary and nonbinary structures, the associated parent and truncation maps, and a tau function, and to apply these structures to the construction of solitonic traveling-wave solutions of a finite-field box--ball system (BBS).

The words without the digit \(\mathtt{2}\) form a binary sublanguage, and the increasing enumeration of their values is the OEIS sequence A005187. The increasing enumeration of the values of the remaining words is the OEIS sequence A055938. We construct a local successor transformation on the latter sublanguage. Its orbit enumerates A055938 in increasing order and gives a self-referential difference rule.

A080578 starts with \(a(1)=1\). For \(n\ge2\), its increment is \(1\) if \(n\) has already occurred among the preceding terms, and is \(3\) otherwise. In the terminology of Fokkink and Joshi, A080578 is the Cloitre \((0,1,1,3)\)-hiccup sequence \cite{FokkinkJoshi2026}. The entry for A080578 in the On-Line Encyclopedia of Integer Sequences (OEIS) records
\[
\Aeight(n)=\Afivefive(n-1)+2
\qquad(n\ge2)
\]
as conjectural \cite{OEISA080578}. We prove this relation directly from the successor structure of the nonbinary part of the Mersenne language.

The same OEIS entry records a relation between A080578 and the Conolly sequence, attributed to Cloitre \cite{OEISA080578}. Combining this relation with the A080578--A055938 relation proved below gives a formula for the increasing enumeration of the nonbinary side in terms of the Conolly sequence. Jackson and Ruskey related the first-occurrence structure of the Conolly sequence to A005187 through a binary-tree model \cite{JacksonRuskey2006}. In the present setting, this connection is expressed by identifying a shifted Conolly sequence with the counting function of the binary side. Thus the Conolly sequence describes both the increasing enumeration of the nonbinary side and the counting function of the binary side. We also give a frequency-based proof of the counting identity that avoids the relation with A080578. This two-sided interpretation is the point at which the sequence structure enters the digit dynamics: the binary-side counting function defines the parent map, which is conjugate to deletion of the lowest canonical skew-binary digit.

The binary side defines an intrinsic counting function \(Z\). The associated parent map is the value-side form of digit deletion. Its iterates successively delete the lower digits of the Mersenne representation. The Mersenne digits can be recovered by evaluating \(Z\) along the parent orbit. We also determine the plateau lengths of \(Z\) and relate the plateau lengths and endpoint values to A001511 and A011371.

Appending trailing zeros defines a family of truncation thresholds. The intervals between consecutive thresholds are truncation blocks on which the corresponding parent iterate is constant. The truncation remainder is the value of the deleted lower-digit suffix. Its depth transition recovers the next Mersenne digit.

We then introduce the Mersenne tau function as a finite sum of parent iterates. Its first and second row differences recover the parent iterates and the counting function, respectively. A linear combination of four consecutive rows recovers each Mersenne digit. We also prove that the defect between diagonally adjacent second row differences is equal to the upper-sector digit associated with the truncation remainder. In the finite-depth application, this identity determines the core window residues.

In the final part, we revisit the nested finite-depth one-soliton family of the finite-field BBS established in \cite{Yura2014}. We include the velocity-\(2\) solution as the endpoint \(h=0\), but do not consider the velocity-\(0\) or velocity-\(1\) solutions.

The existence of this family was proved in \cite{Yura2014}. Here we give a different reconstruction based on the Mersenne representation. The distinctive feature of the present approach is that the combinatorial structure is used directly to construct the integer-valued profile and to prove the evolution rule; the finite-field solution is obtained only at the final stage by reduction modulo \(3\). We also construct an integer-valued traveling-wave tau function whose front values are given by the Mersenne tau rows.
Figure~\ref{fig:three-soliton-collision} illustrates a numerical evolution involving the three finite-depth profiles corresponding to \(h=0,1,2\). The present paper constructs each one-soliton \(\widehat S\)-profile explicitly and verifies its traveling-wave equation. The corresponding \(\widehat U\)-profile is obtained from \eqref{eq:U-from-S} in the free-propagation region. Multisoliton collisions and collision stability are not studied here.

\begin{figure}[t]
  \centering
  \includegraphics[width=\textwidth]{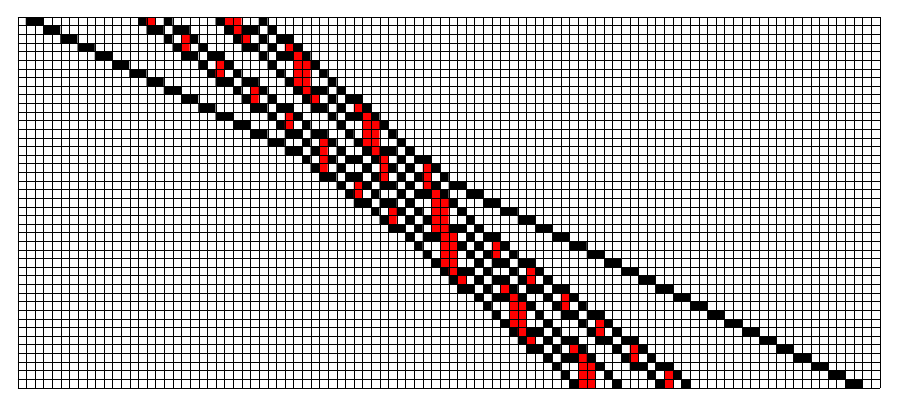}
  \caption{
  A numerical evolution of the finite-field BBS variable
  \(\widehat U_{n}^{t}\in\F_{3}\) for three initially separated
  one-soliton profiles corresponding to \(h=0,1,2\).
  Space increases from left to right and time from top to bottom.
  The field values \(0\), \(1\), and \(2\) are shown in white, black,
  and red, respectively. The corresponding velocities are \(2\),
  \(4/3\), and \(8/7\). In this computation, the three profiles
  reappear after their collisions.
  }
  \label{fig:three-soliton-collision}
\end{figure}

The main results of this paper are as follows. First, the successor structure of the nonbinary Mersenne language gives a direct proof of the relation between A055938 and A080578 that is recorded as conjectural in the OEIS entry for A080578. It also gives a direct connection with the Conolly sequence. Second, the counting function and the associated parent, truncation, and tau-function structures provide digit-reconstruction and diagonal-difference formulas for the Mersenne representation. Third, for the finite-depth one-soliton family established in \cite{Yura2014}, these combinatorial structures give a new integer reconstruction and a window-counting proof that the reconstructed profiles satisfy the finite-field traveling-wave equation.

The paper is organized as follows. We first study the successor structure of the nonbinary Mersenne language and its connections with A055938, A080578, and the Conolly sequence. We then introduce the counting function and develop the parent map, digit reconstruction, plateau structure, and truncation blocks. Next, we introduce the Mersenne tau function and establish its row-difference, digit-recovery, and diagonal-defect formulas. We then prove the complement symmetry of the counting function. In the final part, we construct the finite-depth integer profile, prove the window-counting theorem, and reduce the result modulo \(3\) to obtain the finite-field traveling-wave solutions. We also give their tau-function realization and an example at depth \(h=2\).
\section{Mersenne Language}

\subsection{Words and value map}

Let \(\Sigma=\{\mathtt{0},\mathtt{1},\mathtt{2}\}\). We identify each digit symbol with its digit value whenever it appears in an arithmetic expression. For \(m\ge1\), let \(\LLang_{m}\subset\Sigma^{m}\) be the set of words \(x=d_{m}\cdots d_{1}\) satisfying
\[
d_{k}=\mathtt{2}\quad\Longrightarrow\quad d_{k-1}=d_{k-2}=\cdots=d_{1}=\mathtt{0}
\]
for every \(k\). This is the canonical condition for the skew-binary representation of Myers \cite{Myers1983}. In the terminology of the present paper, such a word is called a Mersenne word. In particular, a Mersenne word contains at most one occurrence of the digit \(\mathtt{2}\). We set \(\LLang_{0}:=\{\eps\}\) and \(\LLang:=\bigcup_{m\ge0}\LLang_{m}\), where \(\eps\) denotes the empty word.

For \(x,y\in\LLang\), define \(x\sim y\) if there exist \(a,b\ge0\) such that \(\mathtt{0}^{a}x=\mathtt{0}^{b}y\).
Thus \(x\sim y\) means that \(x\) and \(y\) differ only by leading zeros. Each equivalence class contains a unique word without a leading zero, except that the class of zero contains the empty word \(\eps\) as its canonical representative. Unless a fixed length is specified, we use this canonical representative.

For \(x=d_{m}\cdots d_{1}\), define
\[
C(x):=\sum_{k=1}^{m}d_{k}(2^{k}-1),
\qquad
P(x):=\sum_{k=1}^{m}d_{k},
\qquad
B(x):=\sum_{k=1}^{m}d_{k}2^{k-1}.
\]
We call \(P(x)\) the Mersenne digit sum of \(x\), and call \(B(x)\) the binary companion of \(x\). For \(m=0\), these are empty sums, and hence \(C(\eps)=P(\eps)=B(\eps)=0\). Thus \(\eps\) is the Mersenne representation of \(0\).

Define the deletion map by
\[
\delta(d_{m}\cdots d_{1}):=d_{m}\cdots d_{2},
\qquad
\delta^{\ell}(d_{m}\cdots d_{1})=d_{m}\cdots d_{\ell+1}.
\]
If \(\ell\ge m\), all digits are deleted and the result is \(\eps\). We also set \(\delta(\eps)=\eps\).

\begin{proposition}\label{prop:CBP}
For every Mersenne word \(x\), one has
\[
C(x)=2B(x)-P(x),
\]
and
\[
B(x)=C(x)-C(\delta(x)),
\qquad
P(x)=C(x)-2C(\delta(x)).
\]
\end{proposition}

\begin{proof}
Write \(x=d_{m}\cdots d_{1}\).  Since \(2^{k}-1=2\cdot2^{k-1}-1\), the first identity follows by summing over \(k\).  Also
\[
C(\delta(x))=\sum_{k=2}^{m}d_{k}(2^{k-1}-1),
\]
and hence
\[
C(x)-C(\delta(x))=d_{1}+\sum_{k=2}^{m}d_{k}2^{k-1}=B(x).
\]
The identity for \(P\) follows from \(C=2B-P\).
\end{proof}

Myers proved that every nonnegative integer has a unique canonical skew-binary representation \cite{Myers1983}. The fixed-length bijection below is the corresponding finite-interval formulation and was also used in Proposition~6 of \cite{Yura2014}. We include a short proof adapted to the present notation because the precise finite interval will be used repeatedly below.
\begin{proposition}\label{prop:fixed-bijection}
For every \(m\ge0\), the fixed-length value map
\[
C_{m}:\LLang_{m}\longrightarrow\{0,1,\ldots,2^{m+1}-2\}
\]
is bijective.  Hence every nonnegative integer has a unique Mersenne word after removing leading zeros.
\end{proposition}

\begin{proof}
We argue by induction on \(m\). For \(m=0\), one has
\[
\LLang_{0}=\{\eps\},
\qquad
C(\eps)=0.
\]
Suppose that the assertion holds for length \(m\). A word of length \(m+1\) belongs to exactly one of three sectors according to its highest digit.

If the highest digit is \(\mathtt{0}\), the remaining \(m\) digits form an arbitrary word \(v\in\LLang_{m}\), and
\[
C(\mathtt{0}v)=C(v).
\]
By the induction hypothesis, this sector represents each integer in
\[
0\le r\le2^{m+1}-2
\]
exactly once.

If the highest digit is \(\mathtt{1}\), the remaining digits again form an arbitrary word \(v\in\LLang_{m}\), and
\[
C(\mathtt{1}v)=(2^{m+1}-1)+C(v).
\]
Hence this sector represents each integer in
\[
2^{m+1}-1\le r\le2^{m+2}-3
\]
exactly once.

If the highest digit is \(\mathtt{2}\), the defining condition of the Mersenne language forces all lower digits to be \(\mathtt{0}\). Thus this sector consists only of \(\mathtt{2}\mathtt{0}^{m}\), whose value is
\[
2(2^{m+1}-1)=2^{m+2}-2.
\]
The three ranges are disjoint and together form
\[
\{0,1,\ldots,2^{m+2}-2\}.
\]
Therefore \(C_{m+1}\) is bijective.

Finally, increasing the fixed length only adds leading zeros. Hence, after leading zeros are removed, every nonnegative integer has a unique Mersenne word.
\end{proof}

By Proposition~\ref{prop:fixed-bijection}, every \(r\in\Z_{\ge0}\) has a unique Mersenne representation, denoted by \(C^{-1}(r)\). Its length is allowed to depend on \(r\).

It follows from Proposition~\ref{prop:fixed-bijection} that \(x\sim y\) if and only if \(C(x)=C(y)\) for \(x,y\in\LLang\). Moreover, leading-zero equivalence is compatible with digit deletion and with admissible right concatenation. More precisely, if \(x\sim y\) and \(\ell\ge0\), then \(\delta^{\ell}(x)\sim\delta^{\ell}(y)\). If, in addition, \(v\in\LLang\) and \(xv,yv\in\LLang\), then \(xv\sim yv\).

For \(m\ge0\), define
\[
C_{m}^{-1}\colon\{0,1,\ldots,2^{m+1}-2\}\longrightarrow\LLang_{m}
\]
to be the inverse of the fixed-length bijection \(C_{m}\). The notation \(C^{-1}(r)\) always denotes the canonical representation, whose length depends on \(r\). By contrast, \(C_{m}^{-1}(r)\) denotes the representation of the prescribed length \(m\), with leading zeros included when necessary.

\subsection{The sublanguages}
In accordance with the convention above, variable-length words in \(J\) and \(\overline J\) are always represented canonically unless a fixed length is explicitly specified.
Define
\[
J:=\{x\in\LLang:x\text{ contains the digit }\mathtt{2}\},
\qquad
\overline J:=\LLang\setminus J.
\]
The value set of the binary sublanguage was studied by
Allouche, B\'etr\'ema, and Shallit in connection with the
von Neumann parenthesis sequence
\cite{AlloucheBetremaShallit1989}.
They characterized the integers representable in the form
\[
\sum_{k\ge1}\epsilon_{k}(2^{k}-1),
\qquad
\epsilon_{k}\in\{0,1\},
\]
where only finitely many \(\epsilon_{k}\) are nonzero, and observed
that such a representation, when it exists, is unique. The empty sum
represents \(0\). In the terminology of the present paper, this value
set is precisely \(C(\overline J)\), and its increasing enumeration is
the OEIS sequence \(\Afive\) \cite{OEIS}. The increasing enumeration
of the complementary value set is the OEIS sequence \(\Afivefive\)
\cite{OEIS}.

Our purpose below is to place this known binary-side value set inside
the full canonical skew-binary language and to study it together with
its complementary nonbinary side \(C(J)\). A word in
\(\overline J\) has only the digits \(\mathtt{0}\) and
\(\mathtt{1}\), so its value is a finite sum of distinct Mersenne
weights. Conversely, every such finite sum is represented by a word in
\(\overline J\). By Proposition~\ref{prop:fixed-bijection}, the
value sets \(C(\overline J)\) and \(C(J)\) form a partition of
\(\Z_{\ge0}\). We refer to \(\overline J\) as the binary
sublanguage and to \(C(\overline J)\) as the binary side. Here the
term binary refers to the digit set
\(\{\mathtt{0},\mathtt{1}\}\), while the positional weights
remain \(2^{k}-1\).

Let
\[
\chi_{\overline J}:\Z_{\ge0}\longrightarrow\{0,1\}
\]
be the characteristic function of the binary-side value set \(C(\overline J)\), defined by
\[
\chi_{\overline J}(n)=
\begin{cases}
1,&n\in C(\overline J),\\
0,&n\notin C(\overline J).
\end{cases}
\]
Up to the identification of the letters \(a\) and \(b\) with
\(1\) and \(0\), respectively, the characteristic sequence
\[
\bigl(\chi_{\overline J}(n)\bigr)_{n\ge0}
\]
is the characteristic word studied by Allouche, B\'etr\'ema, and
Shallit \cite{AlloucheBetremaShallit1989}. This sequence is OEIS
\(\Asevenninefivefivenine\) \cite{OEIS}.

We next describe the action of digit deletion on the nonbinary sublanguage. Let \(\delta_{J}:=\delta|_{J}\colon J\longrightarrow\LLang\) be the restriction of the deletion map to \(J\); thus, \(\delta_{J}(x)=\delta(x)\) for \(x\in J\).

\begin{proposition}\label{prop:deltaJ-bijection}
The map \(\delta_{J}:J\to\LLang\) is bijective. Its inverse assigns to each word the unique one-digit extension that belongs to \(J\). Explicitly,
\[
\delta_{J}^{-1}(x)=
\begin{cases}
 x\mathtt{2},& x\in\overline J,\\
 x\mathtt{0},& x\in J.
\end{cases}
\]
\end{proposition}

\begin{proof}
If \(x\in\overline J\), then \(x\mathtt{2}\) is a Mersenne word in \(J\), and \(\delta(x\mathtt{2})=x\).  If \(x\in J\), then the appended digit must be \(\mathtt{0}\), and \(\delta(x\mathtt{0})=x\).  This also gives uniqueness.
\end{proof}

\section{The Successor Structure of A055938}

Recall that the increasing enumeration of the nonbinary-side value set \(C(J)\) is A055938. We write \(C(J)=\{y_{1}<y_{2}<y_{3}<\cdots\}\), where \(y_{n}=\Afivefive(n)\) for \(n\ge1\). Thus, \(y_{n}\) is the \(n\)-th smallest nonbinary-side value. We construct a successor map on \(J\) whose orbit realizes this enumeration.

\subsection{A local transformation}

Every \(x\in J\) contains a unique digit \(\mathtt{2}\), and all digits below it are \(\mathtt{0}\). The digit immediately above it is either \(\mathtt{0}\) or \(\mathtt{1}\). If \(\mathtt{2}\) is the highest digit, we temporarily adjoin one leading zero. The resulting word is equivalent to \(x\) under \(\sim\) and has the same value. After this temporary padding, the word has a unique expression of one of the forms
\[
u\mathtt{0}\mathtt{2}\mathtt{0}^{m},
\qquad
u\mathtt{1}\mathtt{2}\mathtt{0}^{m},
\]
where \(m\ge0\). Define \(\mathcal T:J\to J\) by the corresponding local replacements
\[
u\mathtt{0}\mathtt{2}\mathtt{0}^{m}\longmapsto u\mathtt{1}\mathtt{0}^{m}\mathtt{2},
\]
\[
u\mathtt{1}\mathtt{2}\mathtt{0}^{m}\longmapsto u\mathtt{2}\mathtt{0}^{m+1}.
\]
Both resulting words belong to \(J\). After the replacement, we remove any unnecessary leading zero.
\begin{lemma}[Value increment under the local transformation]\label{lem:T-increments}
For every \(x\in J\),
\[
C(\mathcal T(x))-C(x)
=
\begin{cases}
3,&\text{in the }u\mathtt{0}\mathtt{2}\mathtt{0}^{m}\text{ case},\\
1,&\text{in the }u\mathtt{1}\mathtt{2}\mathtt{0}^{m}\text{ case}.
\end{cases}
\]
\end{lemma}

\begin{proof}
Since equivalent words under \(\sim\) have the same value, it is enough to compare the displayed local parts. In the first case, \(\mathtt{0}\mathtt{2}\mathtt{0}^{m}\longmapsto\mathtt{1}\mathtt{0}^{m}\mathtt{2}\), and the difference is
\[
(2^{m+2}-1)+2(2^{1}-1)-2(2^{m+1}-1)=3.
\]
In the second case, \(\mathtt{1}\mathtt{2}\mathtt{0}^{m}\longmapsto\mathtt{2}\mathtt{0}^{m+1}\), and the difference is
\[
2(2^{m+2}-1)-\bigl((2^{m+2}-1)+2(2^{m+1}-1)\bigr)=1.
\]
The displayed local words have the same length, so the positions and contributions of the digits in \(u\) are unchanged.
\end{proof}

\begin{lemma}\label{lem:intermediate-values}
Let \(x\in J\), and suppose that its temporarily padded form is \(u\mathtt{0}\mathtt{2}\mathtt{0}^{m}\). Then
\[
C^{-1}(C(x)+1)=u\mathtt{1}\mathtt{0}^{m+1},
\]
and
\[
C^{-1}(C(x)+2)=u\mathtt{1}\mathtt{0}^{m}\mathtt{1}.
\]
In particular,
\[
C(x)+1,\ C(x)+2\in C(\overline J).
\]
\end{lemma}

\begin{proof}
The displayed digit \(\mathtt{2}\) is the unique occurrence of \(\mathtt{2}\) in the temporarily padded form, and hence \(u\) contains no digit \(\mathtt{2}\).  Since the temporarily padded word is equivalent to \(x\), the definition of \(C\) gives
\[
C(x)=C(u\mathtt{0}^{m+2})+2(2^{m+1}-1).
\]
Therefore
\[
C(x)+1=C(u\mathtt{0}^{m+2})+(2^{m+2}-1)=C(u\mathtt{1}\mathtt{0}^{m+1}).
\]
Similarly,
\[
C(x)+2=C(u\mathtt{0}^{m+2})+2^{m+2}
=C(u\mathtt{1}\mathtt{0}^{m}\mathtt{1}).
\]
The uniqueness of the Mersenne representation gives the asserted inverse images.  Both words are in \(\overline J\).
\end{proof}

\begin{proposition}\label{prop:successor-orbit}
Let \(S_{0}:=\mathtt{2}\in\LLang_{1}\), and define recursively
\[
S_{n+1}:=\mathcal T(S_{n})
\qquad(n\ge0).
\]
Then
\[
C(S_{n})=y_{n+1}
\]
for every \(n\ge0\).  Moreover,
\[
C(S_{n})=B(S_{n})+n.
\]
\end{proposition}

\begin{proof}
By Lemma~\ref{lem:T-increments}, the map \(\mathcal T\) increases the value by \(1\) or \(3\).  If the increment is \(1\), there is no integer between \(C(x)\) and \(C(\mathcal T(x))\).  If the increment is \(3\), the two intermediate values belong to \(C(\overline J)\) by Lemma~\ref{lem:intermediate-values}.  Hence \(C(\mathcal T(x))\) is the next value in \(C(J)\) after \(C(x)\).  Since \(C(S_{0})=2\) is the smallest element of \(C(J)\), induction gives \(C(S_{n})=y_{n+1}\).

At \(S_{0}=\mathtt{2}\), one has
\[
C(S_{0})-B(S_{0})=0.
\]
For the local replacement
\[
u\mathtt{0}\mathtt{2}\mathtt{0}^{m}\longmapsto u\mathtt{1}\mathtt{0}^{m}\mathtt{2},
\]
the increments are \(\Delta C=3\) and \(\Delta B=2\). For the local replacement
\[
u\mathtt{1}\mathtt{2}\mathtt{0}^{m}\longmapsto u\mathtt{2}\mathtt{0}^{m+1},
\]
the increments are \(\Delta C=1\) and \(\Delta B=0\). Thus \(\Delta(C-B)=1\) in both cases. It follows by induction that
\[
C(S_{n})-B(S_{n})=n,
\]
which proves the second identity.
\end{proof}

Thus, the successor orbit \(S_{0},S_{1},S_{2},\ldots\) is the word-side realization of the increasing sequence \(y_{1},y_{2},y_{3},\ldots\).

\subsection{The coordinate map}

We next introduce a natural integer coordinate on \(J\). For a word in \(J\), we delete its lowest digit and take the value of the resulting Mersenne word. This coordinate will identify the position of the original word in the successor orbit. Define
\[
\Phi:=C\circ\delta_{J}:J\longrightarrow\Z_{\ge0}.
\]

\begin{proposition}\label{prop:coordinate-map}
For every \(n\ge0\),
\[
\Phi(S_{n})=n.
\]
Consequently,
\[
S_{n}=\Phi^{-1}(n)=\delta_{J}^{-1}(C^{-1}(n)).
\]
Moreover,
\[
\Phi(\mathcal T(x))=\Phi(x)+1
\]
for every \(x\in J\).
\end{proposition}
\begin{proof}
The map \(\Phi\) is bijective. Propositions~\ref{prop:CBP} and \ref{prop:successor-orbit} give
\[
\Phi(S_{n})=C(\delta_{J}(S_{n}))=C(S_{n})-B(S_{n})=n.
\]
The formula for \(S_{n}\) follows from bijectivity. Every \(x\in J\) has the unique form \(x=S_{n}\), and hence
\[
\Phi(\mathcal T(x))=\Phi(S_{n+1})=n+1=\Phi(x)+1.
\]
\end{proof}
\begin{remark}
The transformation \(\mathcal T\) is obtained from the ordinary
canonical skew-binary successor through the deletion bijection.
Indeed,
\[
C\bigl(\delta_{J}(\mathcal T(x))\bigr)
=
C\bigl(\delta_{J}(x)\bigr)+1
\qquad(x\in J).
\]
The corresponding local carry rule for canonical skew-binary
representations goes back to Myers \cite{Myers1983}.
The additional point here is that \(\mathcal T\) gives the successor
within \(C(J)\): its change in the \(C\)-value is \(1\) or \(3\), and
in the latter case the two intermediate values belong to
\(C(\overline J)\).
\end{remark}
Let
\[
\mathcal S=\Phi^{-1}:\Z_{\ge0}\longrightarrow J,
\qquad
\mathcal S(n)=S_{n}.
\]
The preceding relations are summarized by the following commutative diagram:
\[
\begin{tikzcd}[column sep=large,row sep=large]
n
  \arrow[r,"{\textstyle C\circ\mathcal S}"]
  \arrow[dr,"{\textstyle\mathcal S}" description]
&
C(S_{n})=y_{n+1}=\Afivefive(n+1)
\\
C^{-1}(n)=\delta_{J}(S_{n})
  \arrow[u,"{\textstyle C}"']
&
S_{n}
  \arrow[l,"{\textstyle\delta_{J}}"]
  \arrow[u,"{\textstyle C}"']
\end{tikzcd}
\]
The lower-left vertex \(C^{-1}(n)=\delta_{J}(S_{n})\) is the canonical Mersenne word representing \(n\). The lower-right vertex \(S_{n}\) is the unique word in \(J\) obtained from \(C^{-1}(n)\) by \(\delta_{J}^{-1}\). In particular, the diagram expresses
\[
C\circ\delta_{J}\circ\mathcal S
=
\operatorname{id}_{\Z_{\ge0}}.
\]
Since \(C\circ\delta_{J}:J\to\Z_{\ge0}\) is bijective, its inverse is the successor-orbit map:
\[
\mathcal S=(C\circ\delta_{J})^{-1}.
\]
\begin{example}\label{ex:coordinate-values}
The first values occurring in the preceding diagram and the corresponding decomposition into the binary and nonbinary sides are shown in Table~\ref{tab:coordinate-values}. Here \(\mathcal S(n)=S_{n}=\delta_{J}^{-1}(C^{-1}(n))\), and \(C(S_{n})=y_{n+1}\). In the third and fourth columns, the value \(n\) is placed in the side to which it belongs. For \(n\ge1\), the difference column displays \(y_{n+1}-y_{n}\): the value \(\textcolor{binarygreen}{\mathbf{3}}\) indicates the binary side, while \(\textcolor{nonbinaryred}{\mathbf{1}}\) indicates the nonbinary side.
\end{example}

\begin{table}[t]
\centering
\caption{The first coordinate values, the decomposition into \(C(\overline J)\) and \(C(J)\), the successor differences, and the successor orbit.}
\label{tab:coordinate-values}
\small
\begin{tabular}{c|c|c|c|c|c|c}
\hline
\(n\) & \(C^{-1}(n)=\delta_{J}(S_{n})\) & \(n\in C(\overline J)\) & \(n\in C(J)\) & \(y_{n+1}-y_{n}\) & \(S_{n}\) & \(C(S_{n})=y_{n+1}\) \\
\hline
0 & \(\eps\) & \(0\) & -- & -- & \(\mathtt{2}\) & \(2\) \\
1 & \(\mathtt{1}\) & \(1\) & -- & \(\textcolor{binarygreen}{\mathbf{3}}\) & \(\mathtt{12}\) & \(5\) \\
2 & \(\mathtt{2}\) & -- & \(2\) & \(\textcolor{nonbinaryred}{\mathbf{1}}\) & \(\mathtt{20}\) & \(6\) \\
3 & \(\mathtt{10}\) & \(3\) & -- & \(\textcolor{binarygreen}{\mathbf{3}}\) & \(\mathtt{102}\) & \(9\) \\
4 & \(\mathtt{11}\) & \(4\) & -- & \(\textcolor{binarygreen}{\mathbf{3}}\) & \(\mathtt{112}\) & \(12\) \\
5 & \(\mathtt{12}\) & -- & \(5\) & \(\textcolor{nonbinaryred}{\mathbf{1}}\) & \(\mathtt{120}\) & \(13\) \\
6 & \(\mathtt{20}\) & -- & \(6\) & \(\textcolor{nonbinaryred}{\mathbf{1}}\) & \(\mathtt{200}\) & \(14\) \\
7 & \(\mathtt{100}\) & \(7\) & -- & \(\textcolor{binarygreen}{\mathbf{3}}\) & \(\mathtt{1002}\) & \(17\) \\
8 & \(\mathtt{101}\) & \(8\) & -- & \(\textcolor{binarygreen}{\mathbf{3}}\) & \(\mathtt{1012}\) & \(20\) \\
9 & \(\mathtt{102}\) & -- & \(9\) & \(\textcolor{nonbinaryred}{\mathbf{1}}\) & \(\mathtt{1020}\) & \(21\) \\
10 & \(\mathtt{110}\) & \(10\) & -- & \(\textcolor{binarygreen}{\mathbf{3}}\) & \(\mathtt{1102}\) & \(24\) \\
11 & \(\mathtt{111}\) & \(11\) & -- & \(\textcolor{binarygreen}{\mathbf{3}}\) & \(\mathtt{1112}\) & \(27\) \\
12 & \(\mathtt{112}\) & -- & \(12\) & \(\textcolor{nonbinaryred}{\mathbf{1}}\) & \(\mathtt{1120}\) & \(28\) \\
13 & \(\mathtt{120}\) & -- & \(13\) & \(\textcolor{nonbinaryred}{\mathbf{1}}\) & \(\mathtt{1200}\) & \(29\) \\
14 & \(\mathtt{200}\) & -- & \(14\) & \(\textcolor{nonbinaryred}{\mathbf{1}}\) & \(\mathtt{2000}\) & \(30\) \\
15 & \(\mathtt{1000}\) & \(15\) & -- & \(\textcolor{binarygreen}{\mathbf{3}}\) & \(\mathtt{10002}\) & \(33\) \\
\hline
\end{tabular}
\end{table}

The second and final columns display the identities
\[
\delta_{J}(S_{n})=C^{-1}(n),
\qquad
C(S_{n})=y_{n+1}.
\]
For \(n\ge1\), the difference column shows that \(y_{n+1}-y_{n}=3\) on the binary side and \(y_{n+1}-y_{n}=1\) on the nonbinary side. This pattern is formulated below as the self-referential difference rule.

\subsection{The self-referential difference rule}

Let \(\chi_{J}\) be the characteristic function of the value set \(C(J)\):
\[
\chi_{J}(n)=
\begin{cases}
1,& n\in C(J),\\
0,& n\notin C(J).
\end{cases}
\]
Since \(C(J)\) and \(C(\overline J)\) form a partition of \(\Z_{\ge0}\), one has
\[
\chi_{J}(n)=1-\chi_{\overline J}(n)
\qquad(n\ge0).
\]

\begin{proposition}\label{prop:A055938-difference}
For every \(n\ge1\),
\[
y_{n+1}-y_{n}=3-2\chi_{J}(n).
\]
\end{proposition}

\begin{proof}
The step from \(S_{n-1}\) to \(S_{n}\) is the step in which the \(\Phi\)-coordinate goes from \(n-1\) to \(n\).  By Proposition~\ref{prop:coordinate-map}, \(S_{n}=\delta_{J}^{-1}(C^{-1}(n))\).  If \(n\in C(J)\), then \(C^{-1}(n)\in J\), and \(\delta_{J}^{-1}\) appends \(\mathtt{0}\). This corresponds to the case in Lemma~\ref{lem:T-increments} in which the \(C\)-value increases by \(1\). If \(n\notin C(J)\), then \(C^{-1}(n)\in\overline J\), and \(\delta_{J}^{-1}\) appends \(\mathtt{2}\). This corresponds to the case in Lemma~\ref{lem:T-increments} in which the \(C\)-value increases by \(3\).  Hence
\[
y_{n+1}-y_{n}=
\begin{cases}
1,& n\in C(J),\\
3,& n\notin C(J),
\end{cases}
\]
which is the asserted formula.
\end{proof}

\begin{corollary}\label{cor:A055938-counting}
For \(1\le a<b\),
\[
\sum_{k=a}^{b-1}\chi_{J}(k)=\frac{3}{2}(b-a)-\frac{y_{b}}{2}+\frac{y_{a}}{2}.
\]
In particular,
\[
\sum_{k=1}^{n}\chi_{J}(k)=\frac{3n+2-y_{n+1}}{2}.
\]
\end{corollary}

\begin{proof}
Sum Proposition~\ref{prop:A055938-difference} over the indicated range.
\end{proof}
\section{A080578 and the Conolly Sequence}

\subsection{The A080578--A055938 Relation}

We use the definition of A080578 recorded in the OEIS entry \cite{OEISA080578}. The initial value is \(a(1)=1\). For \(n>1\), the increment is \(1\) if \(n\) occurs among the preceding terms, and is \(3\) otherwise.

Recall that \((y_{n})_{n\ge1}\) is A055938, equivalently, the increasing enumeration of \(C(J)\). Define
\[
a_{1}:=1,
\qquad
a_{n}:=y_{n-1}+2\quad(n\ge2).
\]

\begin{theorem}[A080578--A055938 relation]\label{thm:A080578-A055938}
For every \(n\ge2\),
\[
\Aeight(n)=\Afivefive(n-1)+2.
\]
Equivalently,
\[
\Afivefive(n)=\Aeight(n+1)-2\qquad(n\ge1).
\]
\end{theorem}

\begin{proof}
We show that \((a_{n})_{n\ge1}\) satisfies the defining rule of A080578. First, \(a_{1}=1\). Since \(y_{1}=2\), one has \(a_{2}=4\), and hence \(a_{2}-a_{1}=3\). The integer \(2\) does not occur among the preceding terms, since the only preceding term is \(a_{1}=1\). Thus the defining rule holds for \(n=2\).

Let \(n\ge3\). By Proposition~\ref{prop:A055938-difference},
\[
a_{n}-a_{n-1}=y_{n-1}-y_{n-2}=3-2\chi_{J}(n-2).
\]
We claim that
\[
n\in\{a_{1},\ldots,a_{n-1}\}\quad\Longleftrightarrow\quad n-2\in C(J).
\]
If \(a_{m}=n\) for some \(m<n\), then \(m\ge2\), so \(n-2=y_{m-1}\in C(J)\). Conversely, suppose that \(n-2\in C(J)\). Since \((y_{k})_{k\ge1}\) is the increasing enumeration of \(C(J)\), there is a unique \(k\ge1\) such that \(n-2=y_{k}\). Then \(a_{k+1}=n\). Since \(y_{1}=2\) and \((y_{k})_{k\ge1}\) is a strictly increasing sequence of integers, one has \(y_{k}\ge k+1\). Therefore,
\[
k+1\le y_{k}=n-2<n,
\]
so this occurrence is among the preceding terms. Consequently,
\[
a_{n}-a_{n-1}=
\begin{cases}
1,&n\in\{a_{1},\ldots,a_{n-1}\},\\
3,&n\notin\{a_{1},\ldots,a_{n-1}\}.
\end{cases}
\]
Thus \((a_{n})\) has the initial value and increment rule defining A080578. Since this rule determines the sequence recursively, \(a_{n}=\Aeight(n)\). Finally, \(y_{n-1}=\Afivefive(n-1)\), which proves the theorem. The equivalent formula follows by replacing \(n\) with \(n+1\).
\end{proof}

\begin{remark}\label{rem:OEIS-conjecture}
The OEIS entry for A080578 records the relation \(\Aeight(n)=\Afivefive(n-1)+2\) as conjectural \cite{OEISA080578}. The preceding theorem proves this relation by using the Mersenne successor structure.
\end{remark}

\subsection{The Conolly sequence}

Let \((z_{n})_{n\ge1}\) be the Conolly sequence, defined by
\[
z_{1}=z_{2}=1,
\qquad
z_{n}=z_{n-z_{n-1}}+z_{n-1-z_{n-2}}
\quad(n\ge3).
\]
This is OEIS A046699 \cite{Conolly1989,OEIS}. We use the following relation, attributed to Cloitre in the OEIS entry for A080578 \cite{OEISA080578}:
\[
\frac{\Aeight(n)-n}{2}=\Acon(n)
\qquad(n\ge2).
\]

Combining Theorem~\ref{thm:A080578-A055938} with this relation gives the following formula.
\begin{corollary}[Conolly representation of A055938]\label{cor:A055938-Conolly}
For every \(n\ge1\),
\[
y_{n}=2z_{n+1}+n-1.
\]
\end{corollary}

\begin{proof}
By Theorem~\ref{thm:A080578-A055938}, \(\Aeight(n+1)=y_{n}+2\).  On the other hand,
\[
\frac{\Aeight(n+1)-(n+1)}{2}=z_{n+1}.
\]
Thus \(y_{n}+2=2z_{n+1}+n+1\), which gives the formula.
\end{proof}

\subsection{The Successor Orbit and Its Counting Interpretation}
Recall that \(S_{n}\) is the \(n\)-th word in the successor orbit and
\[
C(S_{n})=y_{n+1}.
\]
Corollary~\ref{cor:A055938-Conolly} therefore gives \(C(S_{n})=2z_{n+2}+n\). Together with Proposition~\ref{prop:successor-orbit}, this gives
\[
B(S_{n})=2z_{n+2},
\qquad
P(S_{n})=2z_{n+2}-n.
\]
Thus the Mersenne digit sum along the successor orbit is expressed by the Conolly sequence.

Combining the difference rule with \(y_{n}=2z_{n+1}+n-1\), we get
\[
2(z_{n+2}-z_{n+1})+1=3-2\chi_{J}(n),
\]
and hence
\[
z_{n+2}-z_{n+1}=1-\chi_{J}(n)=\chi_{\overline J}(n).
\]
Summing from \(1\) to \(n\) gives
\[
\#\{1\le k\le n:k\in C(\overline J)\}=z_{n+2}-1.
\]
This formula gives the counting interpretation used below.

\section{The Intrinsic Counting Function \texorpdfstring{\(Z\)}{Z}}

Define
\[
Z(r)
:=
\#\{0\le k<r:k\in C(\overline J)\}
=
\sum_{k=0}^{r-1}\chi_{\overline J}(k)
\qquad(r\ge0),
\]
where the sum is empty for \(r=0\).

This definition uses only the binary side of the Mersenne language and is independent of the relation with the Conolly sequence. We therefore call \(Z\) the intrinsic counting function of the Mersenne language. The uppercase \(Z\) distinguishes this function from the Conolly sequence \((z_{n})_{n\ge1}\).

By definition,
\[
Z(0)=0,
\qquad
Z(r+1)-Z(r)=\chi_{\overline J}(r)
\quad(r\ge0).
\]

\begin{lemma}[Fixed-length binary rank]\label{lem:binary-rank}
For \(x\in\LLang_{m}\), the number of binary words \(u\in\{\mathtt{0},\mathtt{1}\}^{m}\) with \(C(u)<C(x)\) is \(B(x)\).
\end{lemma}
\begin{proof}
We argue by induction on \(m\). The assertion is immediate for \(m=0\). Write \(x=d_{m}v\).

If \(d_{m}=\mathtt{0}\), the induction hypothesis gives \(B(v)=B(x)\) binary words below \(x\).

If \(d_{m}=\mathtt{1}\), all \(2^{m-1}\) binary words with highest digit \(\mathtt{0}\) lie below \(x\). Among the words with highest digit \(\mathtt{1}\), the suffix must have value less than \(C(v)\). The induction hypothesis therefore gives
\[
2^{m-1}+B(v)=B(x)
\]
binary words below \(x\).

If \(d_{m}=\mathtt{2}\), admissibility gives
\[
x=\mathtt{2}\mathtt{0}^{m-1}.
\]
All \(2^{m}\) binary words of length \(m\) lie below \(x\), and \(B(x)=2^{m}\).
\end{proof}

\begin{proposition}[Binary companion as the intrinsic counting function]\label{prop:B-counting}
For every \(r\ge0\), \(B(C^{-1}(r))=Z(r)\). Consequently, \(P(C^{-1}(r))=2Z(r)-r\).
\end{proposition}
\begin{proof}
The case \(r=0\) is immediate. Let \(x=C^{-1}(r)\) have length \(m\).

By Proposition~\ref{prop:fixed-bijection}, every canonical word of length greater than \(m\) has value at least \(2^{m+1}-1\). On the other hand, \(r=C(x)\le2^{m+1}-2\). Hence every binary-side value less than \(r\) has a representative of length at most \(m\).

After leading zeros are added, these representatives are in one-to-one correspondence with the binary words of length \(m\) whose values are less than \(C(x)\). Lemma~\ref{lem:binary-rank} therefore gives
\[
Z(r)=B(x).
\]
The second identity follows from \(C(x)=2B(x)-P(x)\).
\end{proof}

\begin{corollary}[Conolly representation of \(Z\)]\label{cor:Z-Conolly}
For every \(n\ge1\),
\[
Z(n)=z_{n+1}.
\]
Equivalently,
\[
y_{n}=2Z(n)+n-1.
\]
\end{corollary}
\begin{proof}
Replacing \(n\) by \(n-1\) in the preceding counting formula gives
\[
\#\{1\le k\le n-1:k\in C(\overline J)\}=z_{n+1}-1.
\]
Since \(0\in C(\overline J)\), the definition of \(Z\) gives
\[
Z(n)=z_{n+1}.
\]
The formula for \(y_{n}\) follows from Corollary~\ref{cor:A055938-Conolly}.
\end{proof}

\begin{remark}\label{rem:Z-Conolly-convention}
The function \(Z\) counts the elements of \(C(\overline J)\) below its argument, including \(0\). The Conolly representation is
\[
Z(n)=z_{n+1}
\qquad(n\ge1).
\]
This indexing includes the empty word without an exceptional case.
\end{remark}

\begin{remark}[Relation to the first-occurrence formulation]\label{rem:Jackson-Ruskey}
Jackson and Ruskey related the first-occurrence structure of the Conolly sequence to A005187 through a binary-tree model \cite{JacksonRuskey2006}. In the present notation, this connection is expressed by the counting formula
\[
Z(n)=z_{n+1}.
\]
Their first-occurrence relation is not used in the proof of the A080578--A055938 theorem.
\end{remark}

\section{The Parent Map and Digit Reconstruction}

\subsection{The parent map and digit deletion}

Define
\[
\pi(r):=r-Z(r).
\]
Since
\[
\chi_{J}(k)+\chi_{\overline J}(k)=1,
\]
one has
\[
\pi(r)
=
\#\{0\le k<r:k\in C(J)\}
=
\sum_{k=0}^{r-1}\chi_{J}(k)
\qquad(r\ge0),
\]
where the sum is empty for \(r=0\). In particular,
\[
\pi(0)=0,
\qquad
\pi(r+1)-\pi(r)=\chi_{J}(r)
\quad(r\ge0).
\]

Thus \(\pi(r)\) counts the nonbinary-side values below \(r\). The following proposition identifies \(\pi\) with digit deletion on the value side.

\begin{proposition}[Digit deletion and the parent map]\label{prop:pi-deletion}
For every Mersenne word \(x\),
\[
C(\delta(x))=\pi(C(x)).
\]
Consequently, for every \(r,\ell\ge0\),
\[
C(\delta^{\ell}(C^{-1}(r)))=\pi^{\ell}(r).
\]
\end{proposition}
\begin{proof}
By Propositions~\ref{prop:CBP} and \ref{prop:B-counting},
\[
C(\delta(x))=C(x)-B(x)=C(x)-Z(C(x))=\pi(C(x)).
\]
The iterated identity follows by induction.
\end{proof}
In view of Proposition~\ref{prop:pi-deletion}, we call \(\pi\) the parent map of the Mersenne representation. Its iterates correspond to successive deletion of the lower digits.

The compatibility of digit deletion with the parent map is expressed by the following commutative diagram:
\[
\begin{tikzcd}[column sep=large,row sep=large]
\LLang
  \arrow[r,"{\textstyle\delta}"]
  \arrow[d,"{\textstyle C}"']
&
\LLang
  \arrow[d,"{\textstyle C}"]
\\
\Z_{\ge0}
  \arrow[r,"{\textstyle\pi}"']
&
\Z_{\ge0}
\end{tikzcd}
\]
Thus \(C\circ\delta=\pi\circ C\).

\begin{lemma}\label{lem:fixed-pi}
Let \(h\ge0\), and let \(K^{(h)}=2^{h+1}-1\). If \(0\le\mu\le2K^{(h)}\), then
\[
C_{h+1}^{-1}(\mu)\in\LLang_{h+1}.
\]
Moreover, for \(0\le\ell\le h+1\),
\[
C\left(\delta^{\ell}(C_{h+1}^{-1}(\mu))\right)=\pi^{\ell}(\mu).
\]
\end{lemma}
\begin{proof}
Since \(2K^{(h)}=2^{h+2}-2\), Proposition~\ref{prop:fixed-bijection}, applied with \(m=h+1\), shows that \(C_{h+1}^{-1}(\mu)\in\LLang_{h+1}\) is well-defined. Moreover,
\[
C_{h+1}^{-1}(\mu)\sim C^{-1}(\mu).
\]
Compatibility with digit deletion gives
\[
\delta^{\ell}\left(C_{h+1}^{-1}(\mu)\right)
\sim
\delta^{\ell}\left(C^{-1}(\mu)\right).
\]
Applying \(C\) and using Proposition~\ref{prop:pi-deletion} proves the assertion.
\end{proof}

\subsection{Digit reconstruction from
\texorpdfstring{\(Z\) and \(\pi\)}{Z and pi}}

The parent iterates successively delete the lower digits of the Mersenne representation. The following formula recovers each digit from the counting function along the parent orbit.

\begin{proposition}[Digit reconstruction]\label{prop:digit-recovery}
Let \(C^{-1}(r)=d_{m}\cdots d_{1}\).  Then, for every \(i\ge1\),
\begin{equation}\label{eq:digit-recovery}
d_{i}=Z(\pi^{i-1}(r))-2Z(\pi^{i}(r)),
\end{equation}
where \(d_{i}=0\) for \(i>m\).
\end{proposition}

\begin{proof}
For the lowest digit,
\[
B(C^{-1}(r))=2B(\delta(C^{-1}(r)))+d_{1},
\]
so \(d_{1}=Z(r)-2Z(\pi(r))\). For general \(i\), let \(x_{i-1}=\delta^{i-1}(C^{-1}(r))\). Its lowest digit is \(d_{i}\), while
\[
C(x_{i-1})=\pi^{i-1}(r),\qquad C(\delta(x_{i-1}))=\pi^{i}(r).
\]
Applying the lowest-digit formula proves the assertion. If \(i>m\), both iterates are zero; this also covers \(r=0\).
\end{proof}

Thus the Mersenne digits are determined by \(Z\) along the parent orbit of \(r\). After the Mersenne tau function is introduced, this formula will be rewritten entirely in terms of four consecutive tau rows.

\section{Plateau Structure and the Conolly Frequency Law}

Let
\[
C(\overline J)=\{b_{0}<b_{1}<b_{2}<\cdots\},
\qquad
b_{0}=0,
\]
be the increasing enumeration of the binary-side value set.
By the definition of \(Z\),
\[
b_{m}=\max\{r\ge0:Z(r)=m\}
\qquad(m\ge0).
\]
Thus \(b_{m}\) is the right endpoint of the plateau on which \(Z\) has value \(m\). For \(m\ge1\),
\[
Z(r)=m
\quad\Longleftrightarrow\quad
b_{m-1}<r\le b_{m}.
\]
Hence this plateau contains \(b_{m}-b_{m-1}\) integers. Define
\[
H(m):=b_{m}-b_{m-1}
\qquad(m\ge1).
\]

For each \(m\ge0\), let
\[
w_{m}=C^{-1}(b_{m}).
\]
Since \(b_{m}\) belongs to the binary side, the word \(w_{m}\) contains only the digits \(\mathtt{0}\) and \(\mathtt{1}\). Moreover, Proposition~\ref{prop:B-counting} gives
\[
B(w_{m})=Z(b_{m})=m.
\]
Thus the digits of \(w_{m}\) are precisely the binary digits of \(m\). Using \(C=2B-P\), we obtain
\[
b_{m}=C(w_{m})=2m-P(w_{m}).
\]

For a positive integer \(n\), let \(\nu_{2}(n)\) denote the exponent of \(2\) in its prime factorization.

\begin{proposition}[Plateau lengths and endpoint values]\label{prop:plateau-lengths}
For every \(m\ge1\),
\[
H(m)=b_{m}-b_{m-1}=1+\nu_{2}(m)=\nu_{2}(2m)=\Aonefiveoneone(m).
\]
Moreover,
\[
\pi(b_{m})=m-P(w_{m})=\nu_{2}(m!)=\Aoneonethreesevenone(m).
\]
\end{proposition}

\begin{proof}
Using
\[
b_{m}=2m-P(w_{m}),
\]
we obtain
\[
H(m)=2-\bigl(P(w_{m})-P(w_{m-1})\bigr).
\]
Let
\[
e=\nu_{2}(m).
\]
Since \(B(w_{m-1})=m-1\) and \(B(w_{m})=m\), the passage from \(w_{m-1}\) to \(w_{m}\), after adding one leading zero if necessary, changes \(e\) trailing digits \(\mathtt{1}\) to \(\mathtt{0}\) and changes the next digit \(\mathtt{0}\) to \(\mathtt{1}\). Therefore,
\[
P(w_{m})-P(w_{m-1})=1-e,
\]
and hence
\[
H(m)=1+\nu_{2}(m)=\nu_{2}(2m).
\]
Since \(Z(b_{m})=m\), we also have
\[
\pi(b_{m})
=
b_{m}-Z(b_{m})
=
m-P(w_{m}).
\]
The quantity \(P(w_{m})\) is the number of digits \(\mathtt{1}\) in the binary representation of \(m\). Legendre's formula therefore gives
\[
m-P(w_{m})=\nu_{2}(m!),
\]
and the result follows \cite{HardyWright}.
\end{proof}

\begin{remark}[Relation to the von Neumann characteristic word]
The plateau-length formula
\[
H(m)=1+\nu_{2}(m)
\]
is equivalent to the block decomposition of the characteristic word
obtained by Allouche, B\'etr\'ema, and Shallit
\cite{AlloucheBetremaShallit1989}.
In their notation, the characteristic word of the binary-side value
set has the form
\[
a^{2}b^{f(1)}a^{2}b^{f(2)}a^{2}b^{f(3)}\cdots,
\qquad
f(j)=1+\nu_{2}(j).
\]
We include the preceding direct proof because it expresses the same
formula in terms of the binary companion and the plateau endpoints of
the intrinsic counting function \(Z\). It also gives simultaneously
the parent-map identity
\[
\pi(b_{m})=\nu_{2}(m!).
\]
\end{remark}

\begin{remark}[A frequency-based proof avoiding A080578]\label{rem:Conolly-frequency}
Although Corollary~\ref{cor:Z-Conolly} was derived above from the OEIS relation involving A080578, it also admits the following intrinsic proof, which uses only the plateau structure of \(Z\) and the frequency characterization of the Conolly sequence. For every \(m\ge1\), the value \(m\) occurs in
\[
Z(1),Z(2),Z(3),\ldots
\]
exactly
\[
H(m)=1+\nu_{2}(m)
\]
times.

After its first term is removed, the Conolly sequence has the same frequency property: each positive integer \(m\) occurs exactly \(1+\nu_{2}(m)\) times \cite{EricksonEtAl2012}. Both sequences are nondecreasing. Hence they coincide term by term:
\[
Z(n)=z_{n+1}
\qquad(n\ge1).
\]
\end{remark}

\begin{remark}
The plateau-length formula identifies the first differences of A005187 with the ruler function
\[
m\longmapsto\nu_{2}(2m)=1+\nu_{2}(m).
\]
In the present notation,
\[
H(m)=\Aonefiveoneone(m),
\qquad
\pi(b_{m})=\Aoneonethreesevenone(m).
\]
Thus A001511 gives the plateau lengths of \(Z\), while A011371 gives the values of \(\pi\) at the right endpoints of these plateaus.
\end{remark}

\section{Sector Recursion and Truncation Blocks}
For a fixed word length, a highest-digit sector is the set of Mersenne words with a prescribed highest digit. When restricted to binary words, the sectors with highest digit \(\mathtt{0}\) and \(\mathtt{1}\) are called the lower and upper binary sectors, respectively.
\begin{lemma}[Sector recursion for \(Z\)]\label{lem:Z-sector}
Let \(m\ge1\), and let \(A_{m}=2^{m}-1\). Then, for \(0\le s\le2^{m}-2\),
\[
Z(A_{m}+s)=2^{m-1}+Z(s).
\]
Moreover,
\[
Z(2A_{m})=2^{m}.
\]
\end{lemma}
\begin{proof}
The binary words of length \(m\) in the lower binary sector have highest digit \(\mathtt{0}\) and contribute \(2^{m-1}\) values below \(A_{m}+s\). A word in the upper binary sector has the form \(\mathtt{1}u\), and
\[
C(\mathtt{1}u)<A_{m}+s\quad\Longleftrightarrow\quad C(u)<s.
\]
The upper binary sector therefore contributes \(Z(s)\) values. At the endpoint \(2A_{m}=C(\mathtt{2}\mathtt{0}^{m-1})\), all \(2^{m}\) binary words of length \(m\) lie below it.
\end{proof}
For \(\ell\ge0\), define
\[
T_{\ell}(r):=r+2(2^{\ell}-1)Z(r).
\]
\begin{lemma}[Trailing-zero shift]\label{lem:zero-shift}
For \(\ell,r\ge0\),
\[
T_{\ell}(r)=C(C^{-1}(r)\mathtt{0}^{\ell}),
\qquad
C^{-1}(T_{\ell}(r))\sim C^{-1}(r)\mathtt{0}^{\ell}.
\]
Consequently,
\[
T_{a}(T_{b}(r))=T_{a+b}(r),
\qquad
Z(T_{\ell}(r))=2^{\ell}Z(r),
\]
and, for \(\ell\ge1\),
\[
\pi(T_{\ell}(r))=T_{\ell-1}(r).
\]
More generally,
\[
\pi^{a}(T_{\ell}(r))=T_{\ell-a}(r)
\qquad(0\le a\le \ell).
\]
\end{lemma}
\begin{proof}
Let \(x=C^{-1}(r)\). Appending \(\ell\) zeros gives
\[
C(x\mathtt{0}^{\ell})=2^{\ell}C(x)+(2^{\ell}-1)P(x).
\]
Since \(P(x)=2Z(r)-r\),
\[
C(x\mathtt{0}^{\ell})
=2^{\ell}r+(2^{\ell}-1)(2Z(r)-r)
=r+2(2^{\ell}-1)Z(r)
=T_{\ell}(r).
\]
The characterization of leading-zero equivalence by the value map gives the second relation. Appending zeros twice gives the composition law, and
\[
B(x\mathtt{0}^{\ell})=2^{\ell}B(x)
\]
gives the identity for \(Z\). Finally, deletion of one appended zero gives \(\pi(T_{\ell}(r))=T_{\ell-1}(r)\), and iteration gives the last formula.
\end{proof}
\begin{corollary}[Truncation-block length]\label{cor:block-length}
For \(\ell,r\ge0\),
\[
T_{\ell}(r+1)-T_{\ell}(r)
=1+2(2^{\ell}-1)\chi_{\overline J}(r)
=
\begin{cases}
2^{\ell+1}-1,&r\in C(\overline J),\\
1,&r\in C(J).
\end{cases}
\]
In particular, \(T_{\ell}\) is strictly increasing.
\end{corollary}
\begin{proof}
The result follows from
\[
Z(r+1)-Z(r)=\chi_{\overline J}(r).
\]
\end{proof}
We call \([T_{\ell}(r),T_{\ell}(r+1))\) the \(\ell\)-truncation block over \(r\).
\begin{theorem}[Truncation blocks and parent iterates]\label{thm:truncation-blocks}
For \(\ell,r,\mu\ge0\),
\[
\pi^{\ell}(\mu)=r
\quad\Longleftrightarrow\quad
T_{\ell}(r)\le\mu<T_{\ell}(r+1).
\]
Consequently,
\[
\pi^{\ell}(\mu)=\max\{r:T_{\ell}(r)\le\mu\}.
\]
\end{theorem}
\begin{proof}
Let
\[
u:=C^{-1}(r).
\]
By the deletion--parent conjugacy in Proposition~\ref{prop:pi-deletion}, the equality
\[
\pi^{\ell}(\mu)=r
\]
is equivalent to the existence of a word \(v\) of length \(\ell\) such that
\[
uv\in\LLang,
\qquad
C^{-1}(\mu)\sim uv.
\]
In that case, \(\mu=C(uv)\). We therefore determine the possible values of \(C(uv)\) as \(v\) ranges over the length-\(\ell\) suffixes for which \(uv\in\LLang\). When a fixed word length is needed, we replace a word by an equivalent word under \(\sim\).

If \(u\in\overline J\), then every word \(v\in\LLang_{\ell}\) may occur as a length-\(\ell\) suffix. For each such \(v\),
\[
C(uv)=C(u\mathtt{0}^{\ell})+C(v)=T_{\ell}(r)+C(v).
\]
By the fixed-length bijection in Proposition~\ref{prop:fixed-bijection}, \(C(v)\) runs once through
\[
0,1,\ldots,2^{\ell+1}-2.
\]
Since Corollary~\ref{cor:block-length} gives
\[
T_{\ell}(r+1)-T_{\ell}(r)=2^{\ell+1}-1,
\]
these values form exactly the integer interval
\[
T_{\ell}(r)\le\mu<T_{\ell}(r+1).
\]
If \(u\in J\), admissibility forces the suffix to be \(\mathtt{0}^{\ell}\). The set of corresponding values is therefore the singleton \(\{T_{\ell}(r)\}\), and Corollary~\ref{cor:block-length} gives \(T_{\ell}(r+1)-T_{\ell}(r)=1\). Thus the same half-open interval description holds in both cases. We have therefore obtained
\[
\pi^{\ell}(\mu)=r
\quad\Longleftrightarrow\quad
T_{\ell}(r)\le\mu<T_{\ell}(r+1).
\]
Since \(T_{\ell}\) is strictly increasing by Corollary~\ref{cor:block-length}, the maximum formula follows.
\end{proof}
\begin{definition}[Mersenne truncation remainder]\label{def:truncation-remainder}
For \(\ell,\mu\ge0\), define
\[
\rho_{\ell}(\mu):=\mu-T_{\ell}(\pi^{\ell}(\mu)).
\]
\end{definition}

\begin{proposition}[Mersenne quotient--remainder decomposition]\label{prop:quotient-remainder}
Let \(\ell,\mu\ge0\), and choose Mersenne words \(u,v\) such that
\[
C^{-1}(\mu)\sim uv,
\]
where \(v\) has length \(\ell\). Then
\[
C(u)=\pi^{\ell}(\mu),
\qquad
C(v)=\rho_{\ell}(\mu),
\]
and
\[
\mu=T_{\ell}(\pi^{\ell}(\mu))+\rho_{\ell}(\mu).
\]
If \(\pi^{\ell}(\mu)\in C(J)\), then \(\rho_{\ell}(\mu)=0\). If \(\pi^{\ell}(\mu)\in C(\overline J)\), then
\[
0\le\rho_{\ell}(\mu)\le2^{\ell+1}-2.
\]
For \(0\le a\le\ell\),
\[
\pi^{a}(\rho_{\ell}(\mu))
=
\rho_{\ell-a}(\pi^{a}(\mu)).
\]
For \(\mu\ge1\),
\[
\pi^{\ell}(\mu)-\pi^{\ell}(\mu-1)
=
\begin{cases}1,&\rho_{\ell}(\mu)=0,\\0,&\rho_{\ell}(\mu)>0\end{cases}.
\]
Finally, write \(C^{-1}(\mu)=d_{m}\cdots d_{1}\), with
\(d_{i}=0\) for \(i>m\). Then
\[
\rho_{\ell+1}(\mu)
=
\rho_{\ell}(\mu)+d_{\ell+1}(2^{\ell+1}-1).
\]
\end{proposition}
\begin{proof}
Since \(C^{-1}(\mu)\sim uv\), one has
\[
\mu=C(uv)=C(u\mathtt{0}^{\ell})+C(v).
\]
The trailing-zero shift in Lemma~\ref{lem:zero-shift} and the deletion--parent conjugacy in Proposition~\ref{prop:pi-deletion} give
\[
C(u\mathtt{0}^{\ell})=T_{\ell}(C(u))=T_{\ell}(\pi^{\ell}(\mu)),
\]
which proves the quotient--remainder decomposition. The range of the remainder follows from admissibility and the fixed-length bijection in Proposition~\ref{prop:fixed-bijection}. Deleting the same lower \(a\) digits from the suffix and from the full word gives
\[
\pi^{a}(\rho_{\ell}(\mu))
=
\rho_{\ell-a}(\pi^{a}(\mu)).
\]

It remains to prove the formula for \(\pi^{\ell}(\mu)-\pi^{\ell}(\mu-1)\). By Theorem~\ref{thm:truncation-blocks}, the function \(\pi^{\ell}\) is constant on each truncation block. If \(\rho_{\ell}(\mu)>0\), then \(\mu-1\) and \(\mu\) lie in the same block, and hence
\[
\pi^{\ell}(\mu)-\pi^{\ell}(\mu-1)=0.
\]
If \(\rho_{\ell}(\mu)=0\), then \(\mu=T_{\ell}(r)\), where \(r=\pi^{\ell}(\mu)\). Since \(\mu\ge1\), one has \(r\ge1\), and \(\mu-1\) is the final point of the preceding block. Therefore
\[
\pi^{\ell}(\mu)=r,
\qquad
\pi^{\ell}(\mu-1)=r-1,
\]
which proves
\[
\pi^{\ell}(\mu)-\pi^{\ell}(\mu-1)
=
\begin{cases}1,&\rho_{\ell}(\mu)=0,\\0,&\rho_{\ell}(\mu)>0\end{cases}.
\]
Finally, increasing the truncation depth from \(\ell\) to \(\ell+1\) adds the digit \(d_{\ell+1}\) with Mersenne weight \(2^{\ell+1}-1\), which gives the last identity.
\end{proof}

\begin{remark}[Nesting of truncation remainders]\label{rem:remainder-nesting}
For \(a,b\ge0\), the lower \(a+b\) digits split into the lower \(b\) digits and the \(a\) digits immediately above them. Hence
\[
\rho_{a+b}(\mu)
=
\rho_{b}(\mu)
+
T_{b}\left(\rho_{a}(\pi^{b}(\mu))\right).
\]
The depth transition in the preceding proposition is the case \(a=1\).
\end{remark}

\begin{corollary}[Row-difference identity]\label{cor:pi-row-difference}
For every \(\ell,\mu\ge0\),
\[
\pi^{\ell}(\mu)-\pi^{\ell+1}(\mu)=Z(\pi^{\ell}(\mu)).
\]
\end{corollary}
\begin{proof}
Apply \(r-\pi(r)=Z(r)\) with \(r=\pi^{\ell}(\mu)\).
\end{proof}

These truncation structures will be used in the finite-depth application to distinguish block beginnings from block interiors and to determine the core window residues.

\section{The Mersenne Tau Function}
\subsection{Row Differences and Digit Recovery}
We encode the parent orbit by a family of tail sums. Their first and second row differences recover the parent iterates and the counting function, respectively. These identities also give a direct reconstruction formula for the Mersenne digits.
For \(\ell\ge-1\), define
\[
G_{\ell}(\mu):=\sum_{a=\ell+1}^{\infty}\pi^{a}(\mu).
\]
For \(\ell=-1\), the sum starts at \(a=0\). If \(C^{-1}(\mu)\) has length \(m\), then \(\pi^{a}(\mu)=0\) for \(a\ge m\), so the sum is finite. We call \(G_{\ell}\) the Mersenne tau function. We regard \(\ell\) as the row index and \(\mu\) as the column index of the array \(G_{\ell}(\mu)\).
\begin{proposition}[Row differences of the Mersenne tau function]\label{prop:tau-row-differences}
For \(\ell\ge0\),
\[
G_{\ell-1}(\mu)-G_{\ell}(\mu)=\pi^{\ell}(\mu),
\]
and
\[
G_{\ell-1}(\mu)-2G_{\ell}(\mu)+G_{\ell+1}(\mu)=Z(\pi^{\ell}(\mu)).
\]
\end{proposition}
\begin{proof}
Subtracting the two finite sums leaves only the term \(\pi^{\ell}(\mu)\), which proves the first identity. The second follows from
\[
\pi^{\ell}(\mu)-\pi^{\ell+1}(\mu)=Z(\pi^{\ell}(\mu)).
\]
\end{proof}
We now combine the digit-reconstruction formula with the second row-difference identity.
\begin{corollary}[Tau-digit recovery]\label{cor:tau-digit-recovery}
Let \(C^{-1}(\mu)=d_{m}\cdots d_{1}\), with \(d_{i}=0\) for \(i>m\). Then, for every \(\ell\ge0\),
\[
d_{\ell+1}
=
G_{\ell-1}(\mu)-4G_{\ell}(\mu)+5G_{\ell+1}(\mu)-2G_{\ell+2}(\mu).
\]
Consequently,
\[
\frac{\rho_{\ell+1}(\mu)-\rho_{\ell}(\mu)}{2^{\ell+1}-1}
=
G_{\ell-1}(\mu)-4G_{\ell}(\mu)+5G_{\ell+1}(\mu)-2G_{\ell+2}(\mu).
\]
\end{corollary}
\begin{proof}
By Proposition~\ref{prop:tau-row-differences},
\[
Z\bigl(\pi^{\ell}(\mu)\bigr)
=
G_{\ell-1}(\mu)-2G_{\ell}(\mu)+G_{\ell+1}(\mu),
\]
and
\[
Z\bigl(\pi^{\ell+1}(\mu)\bigr)
=
G_{\ell}(\mu)-2G_{\ell+1}(\mu)+G_{\ell+2}(\mu).
\]
Applying \eqref{eq:digit-recovery} with \(r=\mu\) and \(i=\ell+1\), we obtain
\[
\begin{aligned}
d_{\ell+1}
&=
Z\bigl(\pi^{\ell}(\mu)\bigr)
-2Z\bigl(\pi^{\ell+1}(\mu)\bigr)\\
&=
G_{\ell-1}(\mu)-4G_{\ell}(\mu)
+5G_{\ell+1}(\mu)-2G_{\ell+2}(\mu).
\end{aligned}
\]
The second identity follows by combining the first identity with the remainder depth-transition formula in Proposition~\ref{prop:quotient-remainder}.
\end{proof}

\subsection{Truncation Boundaries and the Diagonal Tau Defect}
We next compare the second row differences at diagonally adjacent points. The block beginning and the block interior require different descriptions. We first identify the predecessor of a truncation block with the successor sequence.
\begin{lemma}[Predecessor of a truncation block]\label{lem:truncation-block-predecessor}
For every \(\ell\ge1\) and \(r\ge1\),
\[
\pi^{\ell-1}\bigl(T_{\ell}(r)-1\bigr)=y_{r}.
\]
\end{lemma}
\begin{proof}
Let
\[
x:=C^{-1}(r-1).
\]
The integer \(T_{\ell}(r)-1\) is the final point of the \(\ell\)-truncation block over \(r-1\). We determine its word representation according to the side containing \(x\).

Suppose first that \(x\in\overline J\). The block over \(r-1\) is then a full truncation block. Its words are obtained by appending all suffixes \(v\in\LLang_{\ell}\) to \(x\), and their values are
\[
C(xv)=T_{\ell}(r-1)+C(v).
\]
The largest value of a word in \(\LLang_{\ell}\) is \(2^{\ell+1}-2\), represented by \(\mathtt{2}\mathtt{0}^{\ell-1}\). Hence the final word of the block is
\[
x\mathtt{2}\mathtt{0}^{\ell-1},
\]
so that
\[
C^{-1}\bigl(T_{\ell}(r)-1\bigr)
\sim
x\mathtt{2}\mathtt{0}^{\ell-1}.
\]
Deleting the lowest \(\ell-1\) digits leaves \(x\mathtt{2}\).

Suppose next that \(x\in J\). Admissibility then forces the appended length-\(\ell\) suffix to be \(\mathtt{0}^{\ell}\). The block over \(r-1\) has length one, and its unique word is
\[
x\mathtt{0}^{\ell}.
\]
Thus
\[
C^{-1}\bigl(T_{\ell}(r)-1\bigr)
\sim
x\mathtt{0}^{\ell},
\]
and deleting the lowest \(\ell-1\) digits leaves \(x\mathtt{0}\).

By Proposition~\ref{prop:deltaJ-bijection}, the two remaining words are unified as
\[
\delta_{J}^{-1}(x)=
\begin{cases}
x\mathtt{2},&x\in\overline J,\\
x\mathtt{0},&x\in J.
\end{cases}
\]
Therefore,
\[
\delta^{\ell-1}
\left(C^{-1}\bigl(T_{\ell}(r)-1\bigr)\right)
\sim
\delta_{J}^{-1}\bigl(C^{-1}(r-1)\bigr).
\]
Proposition~\ref{prop:coordinate-map} gives
\[
\delta_{J}^{-1}\bigl(C^{-1}(r-1)\bigr)=S_{r-1}.
\]
Applying \(C\), using the deletion--parent conjugacy in Proposition~\ref{prop:pi-deletion} and invariance under leading-zero equivalence, yields
\[
\begin{aligned}
\pi^{\ell-1}\bigl(T_{\ell}(r)-1\bigr)
&=
C\left(
\delta^{\ell-1}
\left(C^{-1}\bigl(T_{\ell}(r)-1\bigr)\right)
\right)\\
&=C(S_{r-1})
=y_{r}.
\end{aligned}
\]
\end{proof}

\begin{corollary}[Boundary counting identity]\label{cor:truncation-boundary-count}
For every \(\ell\ge1\) and \(r\ge1\),
\[
Z\left(\pi^{\ell-1}\bigl(T_{\ell}(r)-1\bigr)\right)=2Z(r).
\]
\end{corollary}
\begin{proof}
Lemma~\ref{lem:truncation-block-predecessor} gives
\[
Z\left(\pi^{\ell-1}\bigl(T_{\ell}(r)-1\bigr)\right)=Z(y_{r}).
\]
Since \(y_{r}=C(S_{r-1})\), Proposition~\ref{prop:B-counting} and Proposition~\ref{prop:successor-orbit} give
\[
Z(y_{r})=B(S_{r-1})=y_{r}-r+1.
\]
Using \(y_{r}=2Z(r)+r-1\) from Corollary~\ref{cor:Z-Conolly}, we obtain \(Z(y_{r})=2Z(r)\).
\end{proof}

By the second row-difference identity in Proposition~\ref{prop:tau-row-differences},
\[
Z\bigl(\pi^{a}(\nu)\bigr)=G_{a-1}(\nu)-2G_{a}(\nu)+G_{a+1}(\nu)
\]
is the second row difference centered at the position \((a,\nu)\) of the array \(G_{a}(\nu)\). Therefore, the two terms in the theorem below are the second row differences centered at
\[
(\ell-1,\mu-1)
\qquad\text{and}\qquad
(\ell,\mu),
\]
respectively. These two positions are diagonally adjacent. Their defect is determined by the truncation remainder.
\begin{theorem}[Diagonal tau defect formula]\label{thm:diagonal-defect}
For \(\ell\ge1\) and \(\mu\ge1\),
\begin{equation}\label{eq:diagonal-defect}
Z\bigl(\pi^{\ell-1}(\mu-1)\bigr)
-2Z\bigl(\pi^{\ell}(\mu)\bigr)
=
\left\lfloor\frac{\rho_{\ell}(\mu)}{2^{\ell}}\right\rfloor.
\end{equation}
\end{theorem}
\begin{proof}
Let
\[
r:=\pi^{\ell}(\mu),
\qquad
\rho:=\rho_{\ell}(\mu).
\]
We distinguish the beginning and the interior of the \(\ell\)-truncation block over \(r\).

Suppose first that \(\rho=0\). By the definition of the truncation remainder,
\[
\rho=0
\quad\Longleftrightarrow\quad
\mu=T_{\ell}(r).
\]
Since \(\mu\ge1\) and \(T_{\ell}(0)=0\), one has \(r\ge1\). Corollary~\ref{cor:truncation-boundary-count} gives
\[
Z\bigl(\pi^{\ell-1}(\mu-1)\bigr)
=2Z(r)
=2Z\bigl(\pi^{\ell}(\mu)\bigr).
\]
This proves \eqref{eq:diagonal-defect} because the right-hand side is zero.

Suppose next that \(\rho>0\). Then \(\mu-1\) remains in the same truncation block as \(\mu\). A block over a value in \(C(J)\) has length one, so necessarily \(r\in C(\overline J)\). Let
\[
u:=C^{-1}(r)\in\overline J.
\]
Since
\[
\mu-1=T_{\ell}(r)+(\rho-1),
\]
the fixed-length bijection in Proposition~\ref{prop:fixed-bijection} gives a unique \(v\in\LLang_{\ell}\) such that
\[
C^{-1}(\mu-1)\sim uv,
\qquad
C(v)=\rho-1.
\]
Write \(v=e v'\), where \(e\) is the highest digit and \(v'\) has length \(\ell-1\). Deleting the lowest \(\ell-1\) digits gives
\[
\pi^{\ell-1}(\mu-1)=C(ue).
\]
The values of length-\(\ell\) words with highest digit \(\mathtt{0}\) range from \(0\) to \(2^{\ell}-2\), while those with highest digit \(\mathtt{1}\) range from \(2^{\ell}-1\) to \(2^{\ell+1}-3\). Since \(C(v)=\rho-1\), it follows that
\[
e=
\begin{cases}
0,&1\le\rho\le2^{\ell}-1,\\
1,&2^{\ell}\le\rho\le2^{\ell+1}-2.
\end{cases}
\]
The highest digit \(\mathtt{2}\) cannot occur because it would give \(C(v)=2^{\ell+1}-2\), whereas \(C(v)=\rho-1\le2^{\ell+1}-3\). Hence
\[
e=\left\lfloor\frac{\rho}{2^{\ell}}\right\rfloor\in\{0,1\}.
\]
Since \(u\) and \(ue\) belong to the binary side, Proposition~\ref{prop:B-counting} gives
\[
\begin{aligned}
Z\bigl(\pi^{\ell-1}(\mu-1)\bigr)
&=B(ue)\\
&=2B(u)+e\\
&=2Z(r)+\left\lfloor\frac{\rho}{2^{\ell}}\right\rfloor.
\end{aligned}
\]
Substituting \(r=\pi^{\ell}(\mu)\) proves \eqref{eq:diagonal-defect}.
\end{proof}

\begin{remark}[Extension to negative rows]\label{rem:negative-tau-rows}
Let \(C^{-1}(\mu)=d_{m}\cdots d_{1}\). A direct summation of the digit contributions shows that, for \(\ell\ge-1\),
\[
G_{\ell}(\mu)=\sum_{i=\max(1,\ell+2)}^{m}d_{i}\left(2^{i-\ell}-i+\ell-1\right).
\]
For \(\ell<-1\), we use the right-hand side as the definition of \(G_{\ell}(\mu)\). This extends \(G_{\ell}\) to every \(\ell\in\Z\). Direct subtraction then gives
\[
G_{\ell-1}(\mu)-G_{\ell}(\mu)=
\begin{cases}
\pi^{\ell}(\mu),&\ell\ge0,\\
T_{-\ell}(\mu),&\ell<0.
\end{cases}
\]
Thus the nonnegative rows correspond to digit deletion, while the negative rows correspond to zero insertion. The negative-row construction is an algebraic extension of the Mersenne tau rows. It does not in general agree with the reflected finite closure used for the finite-field BBS profile. Accordingly, the reflection seam will be treated by complement symmetry rather than by a formal continuation to negative rows.
\end{remark}
\section{Complement Symmetry of \texorpdfstring{\(Z\)}{Z}}
The reflection seam in the finite-depth application will involve two complementary arguments of \(Z\). We prove the required symmetry for binary words of an arbitrary fixed length and then record its finite-depth specialization.
For \(m\ge1\), define \(\Lambda_{m}:=C(\mathtt{1}^{m})+1=2^{m+1}-1-m\).
\begin{lemma}[Binary complement symmetry]\label{lem:Z-complement}
For every \(m\ge1\) and \(0\le j\le\Lambda_{m}\),
\[
Z(j)+Z(\Lambda_{m}-j)=2^{m}.
\]
\end{lemma}
\begin{proof}
After leading zeros are added when necessary, the \(2^{m}\) binary words of length \(m\) represent exactly the binary-side values in \([0,\Lambda_{m}-1]\). Let \(w=d_{m}\cdots d_{1}\), and define
\[
w^{\vee}:=(1-d_{m})\cdots(1-d_{1}).
\]
This map is an involution, and \(C(w^{\vee})=\Lambda_{m}-1-C(w)\). Hence \(C(w^{\vee})<\Lambda_{m}-j\) if and only if \(C(w)\ge j\). Thus the two terms in the assertion count complementary subsets of the binary words of length \(m\).
\end{proof}
\begin{corollary}[Finite-depth complementarity]\label{cor:seam-complementarity}
Let \(h\ge1\), \(K=2^{h+1}-1\), and \(\Lambda=K-h\). Then
\[
Z(j)+Z(\Lambda-j)=2^{h}
\qquad(0\le j\le\Lambda).
\]
\end{corollary}
\begin{proof}
This is Lemma~\ref{lem:Z-complement} with \(m=h\).
\end{proof}
\section{Finite-field BBS and the Finite-depth Profile}

We apply the counting and truncation structures developed above to the nested finite-depth one-soliton family of the finite-field BBS. We first recall the finite-field \(S\)-variable equation to be satisfied by the traveling-wave profile. We then construct a global integer-valued profile \(\sigma\) from the Mersenne parent iterates. The next section studies the inner and outer differences of this profile and proves the window-counting theorem.

Fix \(h\ge0\), and define
\[
K:=2^{h+1}-1,
\qquad
\Omega:=K+1,
\qquad
\Lambda:=K-h.
\]
We also define
\[
R:=3K-(h+1)=2K+\Lambda-1,
\]
\[
d:=(2h+1)K+h-1=2h\Omega+\Lambda-1,
\qquad
N:=d+1=2h\Omega+\Lambda.
\]
The parameter \(h\) is fixed throughout the remainder of the paper. We therefore omit the index \(h\) from quantities that depend on it.

\subsection{The finite-field polynomial \texorpdfstring{\(M\)}{M} and the \texorpdfstring{\(S\)}{S}-variable equation}

The construction in \cite{Yura2014} starts from the max-plus form of the ordinary BBS. Since a finite field has no order compatible with its field operations, the ordinary maximum and minimum functions cannot be used directly over a finite field. However, the transformations among the standard BBS equations use only certain algebraic properties of the maximum operation, in particular commutativity and compatibility with addition. This suggests replacing the maximum operation by a polynomial map over a finite field.

We work over \(\F_{3}\) and set \(L=1\). To distinguish finite-field quantities from the integer-valued functions introduced below, we place a hat on the finite-field variables corresponding to \(U\), \(S\), and \(G\) in \cite{Yura2014}. The polynomial used in this case is
\[
M(a,b)=2(a^{2}+ab+b^{2}+a+b).
\]
The following table compares the polynomial values \(M(a,b)\) in \(\F_{3}\) with the ordinary maximum of the integer representatives \(a,b\in\{0,1,2\}\).
\[
\begin{array}{c|ccc}
M(a,b)\text{ in }\F_{3} \ 
& \ b=0 \ & \ b=1 \ & \ b=2
\\
\hline
a=0 & 0 & 1 & 0
\\
a=1 &  1 & 1 & 2
\\
a=2 & 0 & 2 & 2
\end{array}
\quad
\begin{array}{c|ccc}
\text{ordinary }\max(a,b)  \ 
& \ b=0 \ & \ b=1 \ & \ b=2
\\
\hline
a=0 & 0 & 1 & 2
\\
a=1 & 1 & 1 & 2
\\
a=2 & 2 & 2 & 2
\end{array}
\]
Thus the two tables differ only at
\[
(a,b)=(0,2)
\qquad\text{and}\qquad
(a,b)=(2,0).
\]
The polynomial \(M\) is not the ordinary maximum operation, but it plays the algebraic role required in the finite-field analogue of the BBS.

The finite-field BBS equation~(14) of \cite{Yura2014} is
\begin{equation}\label{eq:ffBBS-U}
\widehat U_{n}^{t+1}
=
-M\left(
-L+\widehat U_{n}^{t},
-\sum_{i=-\infty}^{n-1}
\left(\widehat U_{i}^{t}-\widehat U_{i}^{t+1}\right)
\right)
+M(0,-L).
\end{equation}
For the traveling-wave solutions constructed below, \(\widehat U^{t}\) has finite support for every \(t\). Hence the sum in \eqref{eq:ffBBS-U} contains only finitely many nonzero terms.

The dependent variables are related by
\begin{equation}\label{eq:USG-transform}
\widehat U_{n}^{t}=\widehat S_{n}^{t}-\widehat S_{n-1}^{t},
\qquad
\widehat S_{n}^{t}=\widehat G_{n}^{t-1}-\widehat G_{n}^{t}.
\end{equation}
Here \(\widehat G_{n}^{t}\) denotes the finite-field tau variable in the dependent-variable transformation of \cite{Yura2014}. For the traveling-wave solution constructed in this paper, a corresponding integer-valued lift \(G\) and its reduction \(\widehat G\) will be constructed in Subsection~\ref{subsec:tau-realization}.

Under this transformation, \eqref{eq:ffBBS-U} is represented by the \(S\)-variable equation~(15) of \cite{Yura2014}:
\begin{equation}\label{eq:Yura-S-variable}
\widehat S_{n}^{t}-\widehat S_{n+1}^{t+1}
=
-M(0,-L)
+
M\left(0,\widehat S_{n+1}^{t}-\widehat S_{n}^{t+1}-L\right).
\end{equation}
For the present choice \(L=1\), one has \(M(0,-1)=M(0,2)=0\) in \(\F_{3}\), and hence \eqref{eq:Yura-S-variable} becomes
\begin{equation}\label{eq:S-plaquette}
\widehat S_{n}^{t}-\widehat S_{n+1}^{t+1}
=
M\left(0,\widehat S_{n+1}^{t}-\widehat S_{n}^{t+1}-1\right).
\end{equation}
Only the row \(a=0\) of the \(M\)-table is used below. For \(b\in\F_{3}\),
\[
M(0,b-1)=1
\quad\Longleftrightarrow\quad
b=2
\quad\text{in }\F_{3}.
\]
Thus \(M(0,b-1)=1\) exactly when \(b=2\) in \(\F_{3}\). This detection property reduces the finite-field traveling-wave equation to an integer window-counting rule.

We seek a traveling-wave solution with spatial parameter \(K\) and temporal shift parameter \(\Omega\). We therefore introduce the traveling-wave coordinate
\[
\xi=Kn-\Omega t
\]
and consider a profile of the form
\[
\widehat S_{n}^{t}=\widehat S(\xi).
\]
The corresponding finite-field BBS profile is related to \(\widehat S\) by
\begin{equation}\label{eq:U-from-S}
\widehat U(\xi)=\widehat S(\xi)-\widehat S(\xi-K).
\end{equation}
Under the traveling-wave substitution, the four vertices in \eqref{eq:S-plaquette} have the coordinates
\[
\begin{aligned}
\widehat S_{n}^{t}&=\widehat S(\xi),&
\widehat S_{n+1}^{t}&=\widehat S(\xi+K),\\
\widehat S_{n}^{t+1}&=\widehat S(\xi-\Omega),&
\widehat S_{n+1}^{t+1}&=\widehat S(\xi-1).
\end{aligned}
\]
Here \(K-\Omega=-1\). Therefore the traveling-wave equation to be verified is
\begin{equation}\label{eq:wave-S}
\widehat S(\xi)-\widehat S(\xi-1)
=
M\left(
0,
\widehat S(\xi+K)-\widehat S(\xi-\Omega)-1
\right).
\end{equation}
Since \(\Omega=K+1\), this is equation~(30) of \cite{Yura2014}.

Our purpose is to verify that the reduction modulo \(3\) of the integer profile constructed below satisfies \eqref{eq:wave-S}. We first construct this profile directly from the Mersenne parent iterates.

\subsection{Mersenne Construction of the Front Profile}

Theorem~5 of \cite{Yura2014} gives a finite-depth one-soliton family with
\[
K=2^{h+1}-1,
\qquad
\Omega=2^{h+1},
\qquad
h\ge1.
\]
For \(h=0\), the same parameters give the velocity-\(2\) case in Theorem~4 of \cite{Yura2014}. We include this endpoint in the construction below.

By Lemma~\ref{lem:fixed-pi}, for
\[
0\le\ell\le h+1
\qquad\text{and}\qquad
0\le\mu\le2K,
\]
one has
\[
\pi^{\ell}(\mu)
=
C\left(
\delta^{\ell}
\left(
C_{h+1}^{-1}(\mu)
\right)
\right).
\]
Thus the parent iterates \(\pi^{\ell}\) give the value-side form of the fixed-length Mersenne truncations. We use these iterates to define the front profile for the parameters \(K\) and \(\Omega\).

We now construct the integer-valued profile \(\sigma:\Z\to\Z\) directly from the Mersenne parent iterates. First,
\[
\sigma(\xi):=0\qquad(\xi<0).
\]

We call
\[
0\le\xi<(h+1)\Omega
\]
the front range. For each \(0\le\ell\le h\), consider the block
\[
(h-\ell)\Omega
\le\xi\le
(h-\ell)\Omega+K.
\]
Since \(\Omega=K+1\), these \(h+1\) blocks are consecutive and disjoint, and they cover the front range. Hence every integer in this range has a unique representation
\[
\xi=(h-\ell)\Omega+j,
\qquad
0\le \ell\le h,
\qquad
0\le j\le K.
\]
We define the integer profile on this range by
\begin{equation}\label{eq:front-profile}
\sigma(\xi):=\pi^{\ell}(K+j)
\quad \text{for }\xi=(h-\ell)\Omega+j,
\quad 0\le\ell\le h,
\quad 0\le j\le K.
\end{equation}
Equivalently,
\[
\sigma((h-\ell)\Omega+j)=C\left(\delta^{\ell}(C_{h+1}^{-1}(K+j))\right).
\]
\begin{lemma}[Endpoint parent values]\label{lem:endpoint-parent-values}
For \(0\le\ell\le h+1\),
\[
\pi^{\ell}(K)=2^{h+1-\ell}-1,
\qquad
\pi^{\ell}(2K)=2(2^{h+1-\ell}-1).
\]
\end{lemma}
\begin{proof}
The fixed-length representations are
\[
C_{h+1}^{-1}(K)=\mathtt{1}\mathtt{0}^{h},
\qquad
C_{h+1}^{-1}(2K)=\mathtt{2}\mathtt{0}^{h}.
\]
For \(0\le\ell\le h\), deleting the lower \(\ell\) digits leaves \(\mathtt{1}\mathtt{0}^{h-\ell}\) and \(\mathtt{2}\mathtt{0}^{h-\ell}\), respectively, which have the stated values. For \(\ell=h+1\), both words are deleted completely and both values are zero.
\end{proof}

\subsection{Reflected Closure}

We extend the front profile by reflection. For \((h+1)\Omega\le\xi\le d\), define
\[
\sigma(\xi):=R-\sigma(d-\xi).
\]
Here \(d-\xi\) belongs to the front range, so the right-hand side has already been defined. We call this extension the reflected closure of the front profile. Finally, define
\[
\sigma(\xi):=R\qquad(\xi>d).
\]
When \(h=0\), this exterior definition agrees with the front value \(\sigma(1)=2\).

\begin{lemma}[Front--reflection compatibility]\label{lem:front-reflection}
The front profile is compatible with reflection on the overlap of the front range with its image under \(\xi\mapsto d-\xi\). Consequently, after reflected closure,
\[
\sigma(\xi)+\sigma(d-\xi)=R
\qquad
(\xi\in\Z).
\]
\end{lemma}
\begin{proof}
For \(h=0\), one has \(d=0\), \(R=2\), and \(\sigma(0)=1\), and the identity follows.

Let \(h\ge1\). Since \(R=2K+\Lambda-1\) and \(d=2h\Omega+\Lambda-1\), the front range and its reflection overlap on
\[
(h-1)\Omega+\Lambda
\le\xi\le
(h+1)\Omega-1.
\]
It is enough to prove the identity on one member of each reflection pair.

First, let \(\xi=h\Omega+j\), where \(0\le j<\Lambda\). Then
\[
d-\xi=h\Omega+(\Lambda-1-j).
\]
Both points belong to the top front row, and hence
\[
\begin{aligned}
\sigma(\xi)+\sigma(d-\xi)
&=(K+j)+(K+\Lambda-1-j)\\
&=2K+\Lambda-1
=R.
\end{aligned}
\]

Next, let \(\xi=(h-1)\Omega+j\), where \(\Lambda\le j\le K\). Since \(\Lambda-1=C(\mathtt{1}^{h})\) is the largest binary-side value represented by a word of length \(h\), while \(K=C(\mathtt{1}\mathtt{0}^{h})\) is the next binary-side value, one has \(Z(j)=2^{h}\) for \(\Lambda\le j\le K\). The sector recursion in Lemma~\ref{lem:Z-sector} for \(j<K\), together with \(Z(2K)=K+1\) at \(j=K\), gives
\[
Z(K+j)=K+1
\qquad
(\Lambda\le j\le K).
\]
Consequently, \(\pi(K+j)=j-1\). On the other hand,
\[
d-\xi=h\Omega+(K+\Lambda-j),
\]
and therefore
\[
\begin{aligned}
\sigma(\xi)+\sigma(d-\xi)
&=(j-1)+(2K+\Lambda-j)\\
&=2K+\Lambda-1
=R.
\end{aligned}
\]
The first range covers the reflection pairs for which both points lie in the top front row \((\ell=0)\). The second range covers the remaining reflection pairs in the overlap, for which one point lies in the \(\ell=1\) front row and the other lies in the top front row. Thus the two ranges together contain at least one member of every reflection pair in the overlap. Since the identity is symmetric under the interchange of \(\xi\) and \(d-\xi\), it follows for the other member of each pair, and it holds on the reflected range by definition. In the exterior ranges, the paired values are \(0\) and \(R\). This proves the identity on \(\Z\).
\end{proof}

The reflected closure completes the construction of the global integer profile \(\sigma:\Z\to\Z\). For \(h\ge1\), its definition is summarized by
\[
\sigma(\xi)=
\begin{cases}
0,
& \xi<0,\\
\pi^{\ell}(K+j),
& \xi=(h-\ell)\Omega+j,\quad 0\le\ell\le h,\quad 0\le j\le K,\\
R-\sigma(d-\xi),
& (h+1)\Omega\le\xi\le d,\\
R,
& \xi>d.
\end{cases}
\]
In the reflected range, \(d-\xi\) belongs to the front range, so the third line uses only values already defined by the second line. For \(h=0\), the same construction gives
\[
\sigma(\xi)=
\begin{cases}
0,&\xi<0,\\
1,&\xi=0,\\
2,&\xi>0.
\end{cases}
\]
In both cases,
\[
\sigma(\xi)+\sigma(d-\xi)=R
\qquad
(\xi\in\Z),
\]
and \(\sigma\) is constant on both exterior half-lines. The next section studies the inner and outer differences of this global integer profile and proves the window-counting theorem.

\section{Window-counting Theorem}\label{sec:window-counting}
We study the inner and outer differences of the global integer profile \(\sigma\). Define
\[
D(\xi):=\sigma(\xi)-\sigma(\xi-1),
\qquad
W(\xi):=\sigma(\xi+K)-\sigma(\xi-\Omega).
\]
Since \(\Omega=K+1\), telescoping gives
\[
W(\xi)=\sum_{\eta=\xi-K}^{\xi+K}D(\eta).
\]
Thus \(W(\xi)\) is the sum of \(D\) over the centered interval \([\xi-K,\xi+K]\cap\Z\). Once \(D\) is shown to be \(\{0,1\}\)-valued, this identity shows that \(W(\xi)\) counts the jumps of \(\sigma\) in this interval. The main result of this section shows that the value of \(D\) is detected by the residue of this centered window sum modulo \(3\).

\begin{theorem}[Finite-depth window-counting theorem]\label{thm:window-counting}
For every \(h\ge0\) and every \(\xi\in\Z\),
\[
D(\xi)\in\{0,1\},
\]
and
\[
D(\xi)=1
\quad\Longleftrightarrow\quad
W(\xi)\equiv2\pmod3.
\]
\end{theorem}

The proof occupies the remainder of this section. We first treat the endpoint \(h=0\) and establish a reflection reduction for \(h\ge1\). We then verify the theorem on the left exterior, the regular core, the core endpoints, and the reflection seam.

\subsection{Reflection Reduction}
We first consider \(h=0\). The explicit profile gives
\[
D(\xi)=
\begin{cases}
1,&\xi=0,1,\\
0,&\text{otherwise}.
\end{cases}
\]
Since \(K=1\), one has \(W(\xi)=D(\xi-1)+D(\xi)+D(\xi+1)\), and hence
\[
W(\xi)=
\begin{cases}
2,&\xi=0,1,\\
1,&\xi=-1,2,\\
0,&\text{otherwise}.
\end{cases}
\]
Therefore, Theorem~\ref{thm:window-counting} holds for \(h=0\). In the remainder of this section, let \(h\ge1\).

\begin{lemma}[Reflection symmetry of \(D\) and \(W\)]\label{lem:window-reflection}
For every \(\xi\in\Z\),
\[
D(N-\xi)=D(\xi),
\qquad
W(N-\xi)=W(\xi).
\]
\end{lemma}
\begin{proof}
By Lemma~\ref{lem:front-reflection} and \(N=d+1\),
\[
\begin{aligned}
D(N-\xi)
&=\sigma(d+1-\xi)-\sigma(d-\xi)\\
&=\bigl(R-\sigma(\xi-1)\bigr)-\bigl(R-\sigma(\xi)\bigr)\\
&=D(\xi).
\end{aligned}
\]
Since \(\Omega=K+1\), one has \(N-\xi+K=d-(\xi-\Omega)\) and \(N-\xi-\Omega=d-(\xi+K)\). Therefore,
\[
\begin{aligned}
W(N-\xi)
&=\sigma(d-(\xi-\Omega))-\sigma(d-(\xi+K))\\
&=\bigl(R-\sigma(\xi-\Omega)\bigr)-\bigl(R-\sigma(\xi+K)\bigr)\\
&=W(\xi).
\end{aligned}
\]
\end{proof}

Since \(N=2h\Omega+\Lambda\),
\[
\left\lfloor\frac{N}{2}\right\rfloor
=
h\Omega+\left\lfloor\frac{\Lambda}{2}\right\rfloor.
\]
By Lemma~\ref{lem:window-reflection}, both assertions of Theorem~\ref{thm:window-counting} are invariant under the reflection \(\xi\mapsto N-\xi\). It is therefore enough to prove them for \(\xi\le\lfloor N/2\rfloor\). The fundamental side is the disjoint union of the left exterior \(\xi<0\), the core \(0\le\xi<h\Omega\), and the reflection seam
\[
h\Omega\le\xi\le h\Omega+\left\lfloor\frac{\Lambda}{2}\right\rfloor.
\]
The core consists of the endpoints \(\xi=(h-\ell)\Omega\), where \(1\le\ell\le h\), and the regular-core points \(\xi=(h-\ell)\Omega+j\), where \(1\le\ell\le h\) and \(1\le j\le K\).

\subsection{The Left Exterior}
Let \(\xi<0\). Then \(D(\xi)=0\). Since \(\xi-\Omega<0\), one has \(W(\xi)=\sigma(\xi+K)\). If \(\xi<-K\), then \(W(\xi)=0\). Suppose that \(-K\le\xi<0\), and let \(j=\xi+K\). Then \(0\le j\le K-1\) and
\[
W(\xi)=\sigma(j)=\pi^{h}(K+j).
\]
Since \(K\le K+j\le2K-1\), the highest digit of the fixed-length Mersenne word representing \(K+j\) is \(\mathtt{1}\). Deleting the lower \(h\) digits gives \(\pi^{h}(K+j)=1\). Hence \(D(\xi)=0\) and \(W(\xi)\in\{0,1\}\) throughout the left exterior. In particular, \(D(\xi)\in\{0,1\}\) and \(W(\xi)\not\equiv2\pmod3\), as required.

\subsection{Core Analysis}
We next consider the core range \(0\le\xi<h\Omega\). We treat the regular-core points first and then the core endpoints.

\begin{proposition}[Regular-core formula]\label{prop:regular-core-window}
Let \(\xi=(h-\ell)\Omega+j\), where \(1\le\ell\le h\) and \(1\le j\le K\), and let \(\mu=K+j\). Then
\[
D(\xi)=
\begin{cases}
1,&\rho_{\ell}(\mu)=0,\\
0,&\rho_{\ell}(\mu)>0,
\end{cases}
\]
and
\[
W(\xi)
=
3Z(\pi^{\ell}(\mu))
+
\left\lfloor\frac{\rho_{\ell}(\mu)}{2^{\ell}}\right\rfloor
-
D(\xi).
\]
In particular,
\[
D(\xi)=1
\quad\Longleftrightarrow\quad
W(\xi)\equiv2\pmod3.
\]
\end{proposition}
\begin{proof}
By the front-profile formula~\eqref{eq:front-profile},
\[
\sigma(\xi)=\pi^{\ell}(\mu),
\qquad
\sigma(\xi-1)=\pi^{\ell}(\mu-1).
\]
Hence
\[
D(\xi)=\pi^{\ell}(\mu)-\pi^{\ell}(\mu-1).
\]
The block-beginning difference formula in Proposition~\ref{prop:quotient-remainder} gives the stated formula for \(D(\xi)\).

Moreover, \(\xi+K=(h-(\ell-1))\Omega+(j-1)\), while \(\xi-\Omega=(h-(\ell+1))\Omega+j\). Thus
\[
W(\xi)=\pi^{\ell-1}(\mu-1)-\pi^{\ell+1}(\mu).
\]
For \(\ell=h\), both \(\sigma(\xi-\Omega)\) and \(\pi^{h+1}(\mu)\) are zero. Adding \(D(\xi)\) and using the row-difference identity in Corollary~\ref{cor:pi-row-difference}, we obtain
\[
\begin{aligned}
W(\xi)+D(\xi)
&=Z(\pi^{\ell-1}(\mu-1))+Z(\pi^{\ell}(\mu))\\
&=3Z(\pi^{\ell}(\mu))
+\left\lfloor\frac{\rho_{\ell}(\mu)}{2^{\ell}}\right\rfloor,
\end{aligned}
\]
where the second equality follows from Theorem~\ref{thm:diagonal-defect}. This proves the formula for \(W(\xi)\).

If \(\rho_{\ell}(\mu)=0\), then \(D(\xi)=1\) and \(W(\xi)\equiv-1\equiv2\pmod3\). If \(\rho_{\ell}(\mu)>0\), then Proposition~\ref{prop:quotient-remainder} implies that \(\pi^{\ell}(\mu)\in C(\overline J)\). Hence the corresponding truncation block is a full block of length \(2^{\ell+1}-1\), and \(1\le\rho_{\ell}(\mu)\le2^{\ell+1}-2\). The floor term is therefore \(0\) or \(1\). In this case \(D(\xi)=0\), so \(W(\xi)\not\equiv2\pmod3\). The equivalence follows.
\end{proof}

\begin{lemma}[Core endpoints]\label{lem:core-endpoints}
Let \(\xi=(h-\ell)\Omega\), where \(1\le\ell\le h\). Then
\[
D(\xi)=1,
\qquad
W(\xi)=3\mathbin{\cdot}2^{h-\ell}-1.
\]
In particular, \(W(\xi)\equiv2\pmod3\).
\end{lemma}
\begin{proof}
The four relevant profile values are
\[
\begin{aligned}
\sigma(\xi)&=\pi^{\ell}(K),&
\sigma(\xi-1)&=\pi^{\ell+1}(2K),\\
\sigma(\xi+K)&=\pi^{\ell}(2K),&
\sigma(\xi-\Omega)&=\pi^{\ell+1}(K).
\end{aligned}
\]
The result follows from Lemma~\ref{lem:endpoint-parent-values}.
\end{proof}

\subsection{Reflection Seam}
It remains to consider \(\xi=h\Omega+j\), where \(0\le j\le\lfloor\Lambda/2\rfloor\).

\begin{lemma}[Reflection-seam formula]\label{lem:seam-self-detection}
For every point in the reflection seam,
\[
D(\xi)=1,
\qquad
W(\xi)=3\mathbin{\cdot}2^{h}-1.
\]
In particular, \(W(\xi)\equiv2\pmod3\).
\end{lemma}
\begin{proof}
For \(j\ge1\), the top front row gives \(D(\xi)=1\). For \(j=0\), one has \(\sigma(h\Omega)=K\) and \(\sigma(h\Omega-1)=\pi(2K)=K-1\), so the same equality holds.

Furthermore, \(\sigma(\xi-\Omega)=\pi(K+j)\). Since
\[
d-(\xi+K)=(h-1)\Omega+\Lambda-j,
\]
the reflection identity in Lemma~\ref{lem:front-reflection} gives \(\sigma(\xi+K)=R-\pi(K+\Lambda-j)\). Using \(R=2K+\Lambda-1\) and \(\pi(r)=r-Z(r)\), we obtain
\[
W(\xi)=Z(K+j)+Z(K+\Lambda-j)-1.
\]
Since \(0\le j\le\Lambda-j\le\Lambda\le K-1\), the sector recursion in Lemma~\ref{lem:Z-sector} applies to both terms. Corollary~\ref{cor:seam-complementarity} therefore gives
\[
\begin{aligned}
W(\xi)
&=2^{h+1}+Z(j)+Z(\Lambda-j)-1\\
&=3\mathbin{\cdot}2^{h}-1.
\end{aligned}
\]
\end{proof}

\subsection{Completion of the Proof}
We now complete the proof of Theorem~\ref{thm:window-counting}.
\begin{proof}
The case \(h=0\) was verified directly in the reflection-reduction subsection. Let \(h\ge1\). On the fundamental side, the left-exterior calculation above, Proposition~\ref{prop:regular-core-window}, Lemma~\ref{lem:core-endpoints}, and Lemma~\ref{lem:seam-self-detection} establish both assertions in the left exterior, the regular core, the core endpoints, and the reflection seam, respectively. These ranges cover every \(\xi\le\lfloor N/2\rfloor\). Lemma~\ref{lem:window-reflection} then extends both assertions to every \(\xi\in\Z\). This completes the proof.
\end{proof}

The \(\{0,1\}\)-valued property of \(D\) gives the following interpretation.

\begin{corollary}[Jump-counting interpretation]\label{cor:jump-counting}
Define
\[
A:=\{\xi\in\Z:D(\xi)=1\}.
\]
Then
\[
\sigma(\xi)=\#\{\eta\in A:\eta\le\xi\},
\qquad
\#A=R,
\]
and
\[
W(\xi)=\#\{\eta\in A:\xi-K\le\eta\le\xi+K\}.
\]
\end{corollary}
\begin{proof}
By Theorem~\ref{thm:window-counting}, \(D(\xi)\) is \(1\) exactly when \(\xi\in A\), and is \(0\) otherwise. Since \(\sigma\) is zero on the left exterior, telescoping gives the first formula. On the right exterior, \(\sigma(\xi)=R\), so \(\#A=R\). The formula for \(W\) follows from the telescoping identity at the beginning of this section.
\end{proof}

\section{Finite-field Traveling-wave Solution}\label{sec:traveling-wave-solution}
\subsection{Finite-field Traveling-wave Profile}
We reduce the global integer profile modulo \(3\) and apply the window-counting theorem.

\begin{corollary}[Finite-field traveling-wave profile]\label{cor:traveling-wave}
Let \(h\ge0\), and let \(\sigma\) be the global integer profile. Define the \(\F_{3}\)-valued profile \(\widehat S\) by
\[
\widehat S(\xi)\equiv\sigma(\xi)\pmod3.
\]
Then \(\widehat S\) satisfies \eqref{eq:wave-S}. Define
\[
\widehat U(\xi):=\widehat S(\xi)-\widehat S(\xi-K),
\qquad
\widehat U_{n}^{t}:=\widehat U(Kn-\Omega t).
\]
Then \(\widehat U_{n}^{t}\) gives the corresponding finite-field BBS solution. Moreover,
\[
\operatorname{supp}\widehat U\subseteq[0,d+K]\cap\Z,
\]
and the velocity is
\[
\frac{\Omega}{K}=\frac{2^{h+1}}{2^{h+1}-1}.
\]
\end{corollary}
\begin{proof}
We reduce Theorem~\ref{thm:window-counting} modulo \(3\). Since \(M(0,b-1)=1\) exactly when \(b=2\) in \(\F_{3}\), Theorem~\ref{thm:window-counting} gives
\[
D(\xi)\equiv M(0,W(\xi)-1)\pmod3.
\]
By the definition of \(\widehat S\),
\[
D(\xi)
\equiv
\widehat S(\xi)-\widehat S(\xi-1)
\pmod3,
\]
and
\[
W(\xi)
\equiv
\widehat S(\xi+K)-\widehat S(\xi-\Omega)
\pmod3.
\]
Therefore, in \(\F_{3}\),
\[
\widehat S(\xi)-\widehat S(\xi-1)
=
M\left(
0,
\widehat S(\xi+K)-\widehat S(\xi-\Omega)-1
\right),
\]
which is \eqref{eq:wave-S}.
For each fixed \(t\), one has \(\widehat S_{n}^{t}=0\) for all sufficiently negative \(n\). Hence, using the relation
\[
\widehat U_{n}^{t}
=
\widehat S_{n}^{t}-\widehat S_{n-1}^{t},
\]
the semi-infinite sum in \eqref{eq:ffBBS-U} telescopes to
\[
\sum_{i=-\infty}^{n-1}
\left(
\widehat U_{i}^{t}-\widehat U_{i}^{t+1}
\right)
=
\widehat S_{n-1}^{t}-\widehat S_{n-1}^{t+1}.
\]
Thus the left boundary normalization required for the dependent-variable transformation is satisfied.
The implication from the \(S\)-variable equation to the \(U\)-variable equation under the first transformation in \eqref{eq:USG-transform}, established in \cite{Yura2014}, therefore shows that the profile \(\widehat U_{n}^{t}\) defined above satisfies the finite-field BBS equation \eqref{eq:ffBBS-U}.

Since \(\sigma(\xi)=0\) for \(\xi<0\) and \(\sigma(\xi)=R\) for \(\xi>d\), the profile \(\widehat S\) is constant on both exterior half-lines. Hence \(\widehat U(\xi)=0\) for \(\xi<0\) or \(\xi>d+K\), which proves the support inclusion. Finally, the coordinate \(\xi=Kn-\Omega t\) gives the velocity \(\Omega/K\).
\end{proof}

Figure~\ref{fig:three-soliton-collision} shows the cases \(h=0,1,2\). In each free-propagation region, the corresponding profile has the form
\[
\widehat U_{n}^{t}=\widehat U(Kn-\Omega t)
\]
and travels with velocity \(\Omega/K\). The collision behavior shown in the figure is not used in the proof.

\subsection{Tau-function Realization}\label{subsec:tau-realization}
\begin{proposition}[Traveling-wave tau function]\label{prop:traveling-tau}
Define
\[
G(\xi):=0
\qquad
(0\le\xi\le K).
\]
Then
\begin{equation}\label{eq:global-traveling-tau}
G(\xi+\Omega)-G(\xi)=\sigma(\xi)
\end{equation}
determines a unique integer-valued function \(G\) on \(\Z\). Moreover,
\begin{equation}\label{eq:tau-front-arrangement}
G((h-\ell)\Omega+j)=G_{\ell}(K+j)
\end{equation}
for \(0\le\ell\le h\) and \(0\le j\le K\).
\end{proposition}
\begin{proof}
Since \(\Omega=K+1\), the interval \(0\le\xi\le K\) contains one representative of each residue class modulo \(\Omega\). Together with these initial values, equation~\eqref{eq:global-traveling-tau} determines \(G\) uniquely in both directions on every residue class.

We prove \eqref{eq:tau-front-arrangement} by descending induction on \(\ell\). For \(\ell=h\), one has \(G(j)=0=G_{h}(K+j)\), because \(\pi^{a}(K+j)=0\) for \(a\ge h+1\). Suppose that \(G((h-\ell)\Omega+j)=G_{\ell}(K+j)\). Then
\[
\begin{aligned}
G((h-(\ell-1))\Omega+j)
&=G((h-\ell)\Omega+j)+\sigma((h-\ell)\Omega+j)\\
&=G_{\ell}(K+j)+\pi^{\ell}(K+j)\\
&=G_{\ell-1}(K+j),
\end{aligned}
\]
where the last equality follows from Proposition~\ref{prop:tau-row-differences}. This proves the front formula~\eqref{eq:tau-front-arrangement}.
\end{proof}

Define \(\widehat G(\xi)\equiv G(\xi)\pmod3\) and \(\widehat G_{n}^{t}:=\widehat G(Kn-\Omega t)\). Reducing \eqref{eq:global-traveling-tau} modulo \(3\), we obtain \(\widehat S(\xi)=\widehat G(\xi+\Omega)-\widehat G(\xi)\), and hence \(\widehat S_{n}^{t}=\widehat G_{n}^{t-1}-\widehat G_{n}^{t}\), in agreement with \eqref{eq:USG-transform}.

\subsection{The Case \texorpdfstring{\(h=2\)}{h=2}}
\begin{example}[A traveling-wave tau function at depth \(h=2\)]\label{ex:tau-h2}
Let \(h=2\), so that \(K=7\), \(\Omega=8\), \(R=18\), and \(d=36\). Proposition~\ref{prop:traveling-tau} gives \(G(\xi)=0\) for \(0\le\xi\le7\). For \(0\le j\le7\), its front formula gives
\[
\begin{aligned}
G(j)&=G_{2}(7+j),\\
G(8+j)&=G_{1}(7+j),\\
G(16+j)&=G_{0}(7+j).
\end{aligned}
\]
Evaluating these rows, we obtain
\[
\begin{aligned}
(G(\xi))_{\xi=0}^{7}&=(0,0,0,0,0,0,0,0),\\
(G(\xi))_{\xi=8}^{15}&=(1,1,1,1,1,1,1,2),\\
(G(\xi))_{\xi=16}^{23}&=(4,4,4,5,5,5,6,8).
\end{aligned}
\]
The corresponding forward \(8\)-differences are
\[
\begin{aligned}
(\sigma(\xi))_{\xi=0}^{7}&=(1,1,1,1,1,1,1,2),\\
(\sigma(\xi))_{\xi=8}^{15}&=(3,3,3,4,4,4,5,6),\\
(\sigma(\xi))_{\xi=16}^{23}&=(7,8,9,10,11,12,13,14).
\end{aligned}
\]
Using \(G(\xi+8)=G(\xi)+\sigma(\xi)\), the next block is
\[
(G(\xi))_{\xi=24}^{31}=(11,12,13,15,16,17,19,22).
\]
The reflected part is obtained from \(\sigma(\xi)+\sigma(36-\xi)=18\), and \(\sigma(\xi)=18\) for \(\xi>36\). Since \(G\) is zero on \(0\le\xi\le7\) and \(\sigma\) vanishes on the left exterior, \(G(\xi)=0\) for \(\xi<0\). On the right exterior, where \(\xi>d=36\), one has \(G(\xi+8)-G(\xi)=18\). Hence \(G(\xi)=\frac94\xi+p(\xi)\), where \(p(\xi+8)=p(\xi)\) for \(\xi>d\). Thus the average slope on the right exterior is \(9/4\). Figure~\ref{fig:tau-h2} shows \(G\) and its forward \(8\)-difference.

For comparison, the standard box--ball one-soliton tau function is \(\max(0,\xi)\) up to normalization \cite{TokihiroEtAl1996}. Both functions are constant on the left. The standard box--ball tau function is linear on the right, whereas \(G\) is affine up to an \(8\)-periodic correction on the right exterior. The transition region of \(G\) is a finite staircase determined by the Mersenne tau rows. This is only a qualitative comparison; no limiting relation is asserted.
\end{example}

\begin{figure}[t]
  \centering
  \begin{minipage}[t]{0.48\textwidth}
    \centering
    \includegraphics[width=\textwidth]{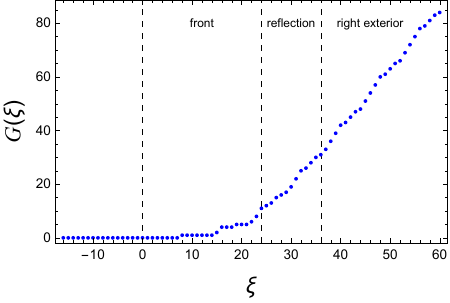}
    \smallskip
    (a)
  \end{minipage}
  \hfill
  \begin{minipage}[t]{0.49\textwidth}
    \centering
    \includegraphics[width=\textwidth]{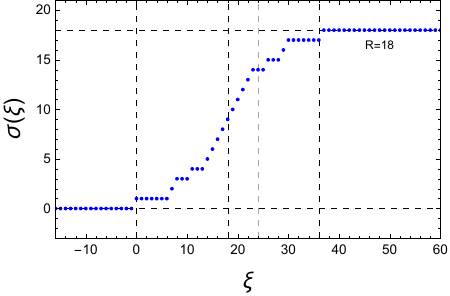}
    \smallskip
    (b)
  \end{minipage}
  \caption{The traveling-wave tau function and its forward \(\Omega\)-difference at depth \(h=2\). (a) \(G(\xi)\); (b) the integer profile \(\sigma(\xi)=G(\xi+8)-G(\xi)\). Here \(K=7\), \(\Omega=8\), \(R=18\), and \(d=36\). The vertical dashed lines at \(\xi=0,24,36\) indicate the
endpoints of the front range and the reflected part. In (b), the additional vertical dashed line at \(\xi=18=d/2\) marks the reflection center of \(\sigma\).}
  \label{fig:tau-h2}
\end{figure}

\section{Conclusion}
We studied the classical canonical skew-binary representation through its binary and nonbinary sides. The Conolly sequence describes both sides: its shift is the intrinsic counting function of the binary side, while the same sequence gives the increasing enumeration of the nonbinary side through A055938. The successor structure of the nonbinary side proves the conjectured relation between A055938 and A080578.

The binary-side counting function induces a parent map conjugate to deletion of the lowest digit. Its iterates organize the nonnegative integers into truncation blocks and lead to the associated truncation remainders, while the Mersenne tau function gives digit-recovery and diagonal-defect formulas. These structures provide an integer reconstruction and a window-counting proof for the known finite-depth one-soliton family of the finite-field BBS. Reduction modulo \(3\) gives the finite-field profiles, and the corresponding integer-valued traveling-wave tau function has front values described by the Mersenne tau rows.

It remains to determine whether this framework extends to interacting or multisoliton solutions. Another direction is to investigate which of these additional structures persist for the \(q\)-integer weights
\[
[k]_{q}=\frac{q^{k}-1}{q-1}.
\]
For integers \(q\ge2\), these are the place values of the pseudo-\(q\)-ary number systems studied by Cameron and Wood \cite{CameronWood1993}. For \(q=3\), the values represented using only the digits \(\mathtt{0},\mathtt{1},\mathtt{2}\), arranged in increasing order, form A004128 \cite{OEIS}. Thus A004128 is the base-\(3\) analogue of the binary-side enumeration A005187. It would also be interesting to determine whether the resulting combinatorial structures admit applications to soliton profiles over finite fields.
\section*{Acknowledgements}
The author thanks Professor Jeffrey Shallit for bringing several relevant references on canonical skew-binary representations and related numeral systems to his attention. 
This work was supported by the Research Institute for Mathematical Sciences, an International Joint Usage/Research Center located in Kyoto University.  This work was also supported by JSPS KAKENHI Grant Number JP23K03233.

\end{document}